\documentclass[onefignum,onetabnum]{siamonline171218}

\usepackage{array,multirow,makecell}
\usepackage{url}
\usepackage{xcolor}
\usepackage{moreverb}
\usepackage{graphicx} 
\usepackage{gastex}
\usepackage{float}
\usepackage{tikz}
\usepackage{pgfplots}
\pgfplotsset{compat=newest}   % This stifles a lot of errors.
\usetikzlibrary{shapes,decorations,arrows,automata,plotmarks,patterns,petri}
\usepackage{dsfont}

\usepackage{tikz-inet}
\usepackage{pgf}
\usepackage{algorithmic}

\usepackage{mybold} 	
\usepackage{mymathbb}

\usepackage{enumitem}
\setlist[enumerate]{leftmargin=.5in}
\setlist[itemize]{leftmargin=.5in}

\newtheorem{assumption}{Assumption}
\crefname{assumption}{Assumption}{Assumptions}

\theoremstyle{plain}

\usepackage{amsopn}

\def\var{{\rm Var }}
\def\q {$\kern15pt$}    % Indentation spaces for slides, algor, programs.
\def\?{\discretionary{}{}{}}  % Same as \- but does not print the - sign
\def\vol{{\rm vol}}

\def\Var{{\rm Var}}
\def\Cov{{\rm Cov}}

\def\MISE{{\rm MISE}}
\def\AMISE{{\rm AMISE}}
\def\ISB{{\rm ISB}}
\def\AISB{{\rm AISB}}
\def\IV{{\rm IV}}
\def\AIV{{\rm AIV}}
\def\tr{{\sf t}}
\def\d{{\rm d}}

\def\ind{{\rm ind}}

\def\cB{\mathcal{B}}
\def\cO{\mathcal{O}}
\def\cS{\mathcal{S}}

\newcommand{\frakv}{{\mathfrak{v}}}

\newif\ifnotes\notestrue

\def\perhaps#1{{\color{gray}#1\color{black}}}

\def\hpierre#1{}
\newcommand{\hmpierre}[1]{{}}

\def\hamal#1{}

\def\hart#1{}

\def\hflorian#1{}

\definecolor{yel}{cmyk}{0.1,0.95,1.0,0.0}
\definecolor{head}{cmyk}{0.1,1.0,0.8,0.15}
\definecolor{vector}{cmyk}{0.0,0.8,1.0,1.0}  %% Boldface (vectors).
\definecolor{tan}{cmyk}{0.30,0.50,0.60,0}
\definecolor{orange}{cmyk}{0.0,0.6,1.0,0.1}
\definecolor{emp}{cmyk}{1.0,0.5,0.0,0}       %  Royal blue
\definecolor{names}{cmyk}{1.0,0.0,1.0,0.14}
\definecolor{pink}{cmyk}{0.0,0.8,0,0}
\definecolor{paleyellow}{cmyk}{0,0,0.6,0.0}
\definecolor{darkyellow}{cmyk}{0,0.2,1.0,0.2}

\pgfplotscreateplotcyclelist{defaultcolorlist}{%
	{blue!95!black,line width=0.9pt,mark=dot*,solid},
	{red!95!black,line width=0.9pt,mark=square*,solid},
	{green!98!black,line width=0.9pt,mark=triangle*,solid},
	{black,mark=star*,solid},
	{brown!85!black,mark=square,dashed},
	{purple!85!black,mark=triangle,dotted}}

\ifpdf
\hypersetup{
  pdftitle={Density estimation by Randomized Quasi-Monte Carlo},
  pdfauthor={Amal Ben Abdellah, Pierre L'Ecuyer, Art B. Owen, Florian Puchhammer}
}
\fi

\headers{Density estimation by Randomized Quasi-Monte Carlo}{A. Ben Abdellah, P. L'Ecuyer,  A. B. Owen, F. Puchhammer.}

\title{ Density estimation by Randomized Quasi-Monte Carlo%
\thanks{ %First version submitted to the editors July 6, 2018; Second version submitted April 29, 2019
\funding{This work has been supported by a Canada Research Chair, an Inria International Chair,
an IVADO Ressearch Grant, and NSERC Discovery Grant number RGPIN-110050 to P. L'Ecuyer.   
A. B. Owen was supported by the US National Science Foundation
under Grants IIS-1837931, DMS-1521145 and DMS-1407397.
The collaboration was also supported by the NSF Grant DMS-1638521 to SAMSI.}}
}

\author{Amal Ben Abdellah 
\thanks{%
	DIRO, University of Montreal, 2920 Chemin de La Tour, Pavillon Aisenstadt, Montreal, QC, H3T 1N8, Canada
	(\email{amal.ben.abdellah@umontreal.ca}, \email{lecuyer@iro.umontreal.ca}, \email{florian.puchhammer@umontreal.ca}).
}
\and 
Pierre L'Ecuyer\footnotemark[2]
\and
Art B. Owen
\thanks{Department of Statistics, Stanford University, Sequoia Hall, 390 Serra Mall, Stanford, CA, 94305-4065, USA
(\email{owen@stanford.edu}).
  }
\and 
Florian Puchhammer\footnotemark[2]
}

\begin{document}

\maketitle

% REQUIRED
\begin{abstract}
We consider the problem of estimating the density of a random variable $X$ that can be sampled exactly by Monte Carlo (MC). We investigate the effectiveness of replacing MC by randomized quasi Monte Carlo (RQMC) or by stratified sampling over the unit cube, to reduce the integrated variance (IV) and the mean integrated square error (MISE) for kernel density estimators. We show theoretically and empirically that the RQMC and stratified estimators can achieve substantial reductions of the IV and the MISE, and even faster convergence rates than MC in some situations, while leaving the bias unchanged. We also show that the variance bounds obtained via a traditional Koksma-Hlawka-type inequality for RQMC are much too loose to be useful when the dimension of the problem exceeds a few units. We describe an alternative way to estimate the IV, a good bandwidth, and the MISE, under RQMC or stratification, and we show empirically that in some situations, the MISE can be reduced significantly even in high-dimensional settings.  
% Density estimation involves a well known bias-variance tradeoff in the choice of a bandwidth parameter $h$. RQMC improves the convergence at any $h$, although the gains diminish when $h$ is reduced to control bias. 
\end{abstract}

% REQUIRED

% REQUIRED

%%%%%%%%%%%%%%%%%%%%%%%%%%%%%%%%%%%%%%%%%%%%%%%%%%%%%%%%%%%%%%%%%%%%%%%%%%%%%%%%%%%%%%%%%%%%%
\section{Introduction}

We are interested in estimating by simulation the density of a random variable
%  are interested in the distribution of 
$X = g(\bU)$ where $\bU = (U_1,\dots,U_{s}) \sim U[0,1]^s$ 
(uniform over the unit hypercube) and $g : [0,1]^s \to\RR$.
We assume that $g(\bu)$ can be computed easily for any $\bu\in [0,1]^s$,
that $X$ has density $f$ (with respect to the Lebesgue measure) over $\RR$
and we want to estimate $f$ over some bounded interval $[a,b]$.
A flurry of stochastic simulation applications fit this framework; see \cite{sASM07a,sLAW14a}, for example.
%  see \cite{sASM07a,sLAW14a,sWSC18a}, for example.
The vector $\bU$ represents the independent uniform random numbers that drive the simulation.

We denote by $\hat f_n$ a density estimator 
based on a sample of size $n$, and we measure the quality of the estimator 
over $[a,b]$ by the \emph{mean integrated square error} (MISE), defined as 
\[
  \MISE = \int_a^b \EE [\hat f_n(x) - f(x)]^2 \d x,
\]
which we want to minimize. The MISE can be decomposed as the sum of the 
\emph{integrated variance} (IV) and the \emph{integrated square bias} (ISB):
\[
  \MISE = \IV + \ISB 
	      = \int_a^b \EE (\hat f_n(x) - \EE [\hat f_n(x)])^2 \d x
	      + \int_a^b (\EE [\hat f_n(x)] - f(x))^2 \d x.
\]
{Minimizing the MISE generally involves a bias-variance tradeoff.}

The density is often estimated by a histogram for visualization, but one 
can do better with more refined techniques, such as a \emph{kernel density estimator} (KDE),
defined as follows.
One selects a \emph{kernel} $k : \RR\to\RR$, and a constant $h > 0$
called the \emph{bandwidth}, which acts as a horizontal stretching factor for the kernel.  
The kernels considered here are smooth probability densities that are symmetric about $0$.
In our experiments, we will use the Gaussian kernel, which is the standard normal density.
%  function $k(x) = \exp(-x^2/2) / \sqrt{2\pi}$.}
Given a sample $X_1,\dots,X_n$, the KDE at $x\in\RR$ is
%  \label{sec:kernels}
\begin{equation}
 \hat f_n(x)     % = \hat f_{{\rm k},n}(x) 
     = \frac{1}{nh} \sum_{i=1}^n  k\left(\frac{x-X_i}{h}\right).   
   \label{eq:kde}
\end{equation}
\hpierre{nondecreasing on $(-\infty,0]$ and nonincreasing on $[0,\infty)$, have a finite mode $k(0)<\infty$,
a finite variance, are infinitely differentiable and have bounded derivatives of all orders.
In all our numerical experiments, we use the Gaussian kernel,
which is the standard normal density function $k(x) = \exp(-x^2/2) / \sqrt{2\pi}$.}
%  It satisfies all the above assumptions.

Density estimation methods such as KDEs were developed for the context where 
an independent sample $X_1,\dots,X_n$ from the unknown density $f$ is given.
Here we assume that we can generate a sample 
%  $X_1,\dots,X_n$ 
of arbitrary size by choosing where to sample.
With \emph{crude Monte Carlo} (MC), we would estimate the density from a sample 
$X_1,\dots,X_n$ of $n$ \emph{independent} realizations of $X$, obtained by simulation.  
Then the analysis is the same as if the data was collected from the real world,
and the standard KDE methodology would apply \cite{tSCO15a}.
In that context, the IV is $\cO(1/nh)$ and the ISB is $\cO(h^4)$, 
so the MISE is $\cO(n^{-4/5})$ if $h$ is chosen optimally.
This is slower than the $\cO(n^{-1})$ canonical rate for the variance when estimating the mean.

Our aim in this paper is to study if, when, and how using \emph{randomized quasi-Monte Carlo} (RQMC)
or \emph{stratification} can provide a KDE with a smaller MISE than with crude MC.
It is well known that when we estimate the mean $\EE[X]$ by the average $\bar X_n = (X_1 + \cdots + X_n)/n$,
under appropriate conditions, using RQMC provides an unbiased estimator whose variance
converges at a faster rate (in $n$) than the MC variance \cite{rDIC10a,vLEC18a,vLEC02a,vOWE97b,vOWE98a}.
This variance bound is easily proved by squaring a worst-case deterministic error bound obtained via
a version of the \emph{Koksma-Hlawka} (KH) inequality, which is a H\"older-type inequality that bounds
the worst-case integration error by a product of the variation of $g$ and the discrepancy of the 
set of points $\bU$ at which $g$ is evaluated.
\hflorian{Sometimes you use ``KH'' as ``Koksma--Hlawka inequality'' and sometimes it stands for ``Koksma--Hlawka'' only.}
Hundreds of papers have studied this.
%  When the dimension exceeds a half dozen or so, it may actually take an excessively large
%  sample size for this square KH bound to be smaller than the MC variance, 
Of course, the faster rate is an asymptotic property and the KH bound may hide a large constant factor, 
so it could happen that this bound is larger than the MC variance for a given $n$.
But in applications, the true RQMC variance is often much smaller than 
both the bound and the MC variance, even for moderate sample sizes.
The bottom line is that RQMC is practically useful in many applications, when estimating the mean by an average.
Stratification of the unit hypercube also provably reduces the variance of $\bar X_n$,
although its applicability degrades quickly with the dimension, 
and it is typically dominated by RQMC when the dimension exceeds 1 or 2 \cite{vLEC18a}.

Since the KDE \cref{eq:kde} at any given point is an average just like the estimator
of an expectation, it seems natural to use RQMC to estimate a density as well,
and to derive variance bounds via the same methods as for the mean estimator.
This was the starting point of this paper.
%  It seems that nobody else has looked at that before us.
At first, we 
% (and other people we talked to) 
thought that the KH inequality would provide bounds on the IV of the KDE that converge faster for RQMC
than for MC, and that a faster convergence rate of the MISE would follow.  But things are not so simple.
%  What happens is that 
The best upper bound on the IV that KH gave us is $\cO(n^{-2+\epsilon} h^{-2s})$ 
for any $\epsilon > 0$, while the ISB remains $\cO(h^4)$ as with MC.
This gives a bound of $\cO(n^{-4/(2+s) + \epsilon})$ on the MISE if we select $h$ to minimize this bound.
%  For comparison, the MISE rate for the KDE under MC is $\cO(n^{-4/5})$.
The unwelcome $h^{-2s}$ factor in the IV bound comes from the increase of the Hardy-Krause variation of each
summand in \cref{eq:kde} as a function of the underlying uniforms when $h$ decreases.
This effect grows exponentially in $s$.
To exploit the smaller power of $n$ in the IV bound to reduce the MISE bound, one must simultaneously 
decrease the ISB.  One can achieve this by taking a smaller $h$, which in turn drastically  increases
the IV bound.  This limits seriously the rate at which the MISE bound can converge.
The resulting rate for the bound beats the MC rate only for $s < 3$.
%  
%  As a result, the KH bounds are not very useful in more than 2 or 3 dimensions.
% The first part of this paper provides the details of this theoretical study.
% Independently of this, 
For a special type of RQMC method, namely a digital net with a nested uniform scramble, 
we also prove that the IV and MISE rates are never worse than for MC.

For the KDE combined with a stratification of the unit hypercube into subcubes, 
which could be seen as a weak form of RQMC, we obtain bounds that 
converge as $\cO(n^{-(s+1)/s} h^{-2})$ for the IV 
and $\cO(n^{-(2/3)(s+1)/s})$ for the MISE.   
The latter beats the MC rate for all $s < 5$.
These bounds are proved using arguments that do not involve KH and they are tight.
We show examples where the IV and the MISE with stratification behave just like the bounds.

These results do not imply that stratification works better than RQMC, 
or that RQMC does not beat MC in more than two dimensions.
The KH bounds are \emph{only upper bounds} and nothing precludes that the true IV and MISE
can be significantly smaller with RQMC than with MC or stratification, 
even if the RQMC variance bound is larger and converges more slowly.
At a minimum, we should test empirically 
%  So we must find another way to see 
how the KDE really behaves in terms of IV and MISE
under RQMC and under stratification.
We also need a procedure to choose a good bandwidth $h$ for the KDE 
with these sampling methods, since it will generally differ from a good $h$ with MC.
We do that in the second half of the paper.
Our aim is to assess empirically the improvements %  in the IV and MISE 
achieved for reasonable sample sizes $n$ in actual simulations.  %  on some examples.
We use a regression model in log scale to estimate 
the IV and the MISE as functions of $h$ and $n$, and the optimal $h$ as a function of $n$.
We find that RQMC often reduces the IV and the MISE significantly, even in more than 3 dimensions,
and that it performs better than stratification.
Sometimes, the convergence rate of the MISE is not improved but there is 
a significant gain in the constant and in the actual MISE.
In all our experiments, the MISE was never larger with RQMC or stratification than with MC.
We prove that this always holds for stratification.
But for RQMC, we think that proving the observed gains in theorems would be very hard, 
hence the importance of testing with diverse numerical examples.

The remainder is organized as follows. 
In \Cref{sec:density-estimators}, we recall the definitions and basic properties of KDEs,
including a strategy to find a good $h$ under MC. 
In \Cref{sec:rqmc-int}, we recall classical error and variance bounds for RQMC integration.
%  and a few common RQMC methods that we use in the paper.
In \Cref{sec:rqmc-iv}, we use classical QMC theory to derive KH bounds 
on the IV and the MISE for a KDE under RQMC, under reasonable assumptions.
In \Cref{sec:strat} we derive IV and MISE bounds for a KDE combined with stratification.
We have bounds that converge at a faster rate than for MC when the dimension is small.
We also show that stratification never increases the IV or MISE compared with MC.
%  In \Cref{sec:regression}, we introduce simple regression models for the IV, ISB, and MISE, over a limited region,
%  and we explain how we estimate the model parameters.
In \Cref{sec:empirical}, we report on numerical experiments in which we estimate and compare 
the true IV and MISE of the KDE with MC, RQMC, and stratification, for various examples.
We also provide a method to find a good bandwidth $h$,
%  for the KDE with RQMC and stratification; 
which is necessary for their effective implementation,
and we use a regression model to capture how the IV and the MISE really behave in the examples.
%  \hpierre{Estimate model parameters, test gof, ...}
We give our conclusions in \Cref{sec:conclusion}.
\hpierre{When the paper is finalized, we plan to add an online-only supplement that will contain additional numerical results, 
goodness-of-fit tests, and results for RQMC and stratification for histograms instead of KDEs.}

{We adopt the usual $\Theta(\cdot)$ notation for the \emph{exact order}: 
$h(n) = \Theta(\varphi(n))$ means that there is some $n_0$ and constants $c_2 > c_1 > 0$
such that for all $n\ge n_0$, $c_1 \le h(n)/\varphi(n) \le c_2$.
This is less restrictive than $h(n) \propto \varphi(n)$. 
Also, $c(n,h) = \cO(\varphi(n,h))$ means that there is a constant $K > 0$ such that 
for all integers $n \ge 1$ and all $h\in (0,1]$, $c(n,h) \le K \varphi(n,h)$.}

%%%%%%%%%%%%%%%%%%%%%%%%%%%%%%%%%%%%%%%%%%%%%%%%%%%%%%%%%%%%%%
\section{Kernel density estimators with MC}
\label{sec:density-estimators}

%%  -- This section only recalls MC results, nothing for RQMC here.

We recall asymptotic properties of the KDE with MC when 
$nh\to \infty$ and $h\to 0$ together.
The details can be found in \cite{tJON96a,tSCO15a,tWAN95a}, for example.
The asymptotic MISE, IV, and ISB in this regime are denoted
AMISE, AIV, and AISB, respectively.
If $\IV(n,h)$ denotes the IV for a given $(n,h)$,
writing $\AIV = \tilde g(n,h)$ for some function $\tilde g$ means that
$\lim_{n\to \infty,\, \tilde g(n,h)\to 0} \IV(n,h) / \tilde g(n,h) = 1$
and similarly for the AMISE and AISB.
{For measurable functions $\psi : \RR\to\RR$, we define 
the roughness functional $R(\psi) = \int_a^b (\psi(x))^2 \d x$ 
% I changed the argument of $\mu_r$ from $k$ to $g$ to be less confusing.
% Art:  g is also used for our function on $\RR^s$ but I can't think of an obvious better letter here
and the ``moments'' 
$\mu_{r}(\psi) = \int_{-\infty}^\infty x^r \psi(x) \d x$, for integers $r\ge 0$.}
% Art: I took out $\dots$. because the period mingles with the dots in a bad way.
{We make the following assumptions in the rest of the paper.}
\hflorian{
Below, we summarize several conditions on $k$ which are implicitly assumed in the rest of the paper.}

\begin{assumption}
 \label{ass:kGeneral}
The kernel $k$ is a probability density function which is symmetric about 0,
nondecreasing on $(-\infty,0]$ and nonincreasing on $[0,\infty)$, 
has a finite mode $k(0)<\infty$ and its second moment is strictly positive and finite.   
Thus, $\mu_0(k)=1$, $\mu_1(k)=0$, and $0<\mu_2(k)<\infty$.
\hpierre{nondecreasing on $(-\infty,0]$ and nonincreasing on $[0,\infty)$, have a finite mode $k(0)<\infty$,
a finite variance, are infinitely differentiable and have bounded derivatives of all orders.
In all our numerical experiments, we use the Gaussian kernel,
which is the standard normal density function $k(x) = \exp(-x^2/2) / \sqrt{2\pi}$.}
\hflorian{In the introduction we wrote positive and bounded \emph{even} moments. I think assuming this for the second moment is enough. In any case, we need to unify this!}
\end{assumption}

\begin{assumption}
 \label{ass:f}
The density $f$ is at least four times differentiable over $[a,b]$ (including at the boundaries)
%   \mpierre{over $(a-\epsilon,\, b+\epsilon)$ for some $\epsilon > 0$,}
and $R(f^{(r)})<\infty$ for $r\le 4$, where $f^{(r)}$ is the $r$th derivative of $f$.
\end{assumption}

\hpierre{In the rest of the paper, whenever we need it, we assume implicitly that 
$\mu_0(k^r) > 0$, $\mu_2(k) > 0$, and $R(f^{(r)}) > 0$, 
where $f^{(r)}$ is the $r$th derivative of the density $f$.}%
\hart{I wrote in the things that I think we need.  I think we don't go beyond $r=4$.
 Some additional minor remarks are made within comments in the LaTeX source.}

% Under these assumptions, 
With MC, we have  % for the KDE in \cref{eq:kde}, we have 
$\AIV = n^{-1} h^{-1} \mu_0(k^2)$  and  % = \Theta(n^{-1}h^{-1})$.
$\AISB = (\mu_2(k))^2 R(f'') h^4/4$.
The AMISE is minimized by taking
$ h^5 = {Q}/{n}$ where $Q := {\mu_0(k^2)}/[{(\mu_2(k))^2 R(f'')}]$,
if $Q$ is well-defined and finite. This gives 
\[
%  \AMISE = (5/4) Q^{-1/5} \mu_0(k^2) n^{-4/5}.
  \AMISE = (5/4) Q^{-1/5} \mu_0(k^2) n^{-4/5}.
\]
% -- plug-in methods.
% The above formulas tell us exactly how the asymptotically optimal $h$ depends on $k$ and $f$.
Thus, finding a good $h$ amounts to finding a good approximation of $R(f'')$.
But since $f$ is precisely the unknown function that we want to estimate, 
this seems to be a circular problem.
However, perhaps surprisingly, a viable approach is to estimate $R(f'')$ 
by estimating $f''$ also via a KDE, integrating its square over $[a,b]$,
and plugging this estimate into the formula for the optimal $h$ \cite{tBER94a,tJON96a,tRAY06a,tSCO15a}.  
To do that, one needs to select a good $h$ to estimate $f''$ by a KDE. 
The asymptotically optimal $h$ depends in turn on $R(f^{(4)})$  
where $f^{(4)}$ is the fourth derivative of $f$.  
Then $R(f^{(4)})$ can be estimated by integrating 
the KDE of $f^{(4)}$ and this goes on ad infinitum.
In practice, one can select an integer $r_0 \ge 1$, get a rough estimate of $R(f^{(r_0+2)})$, 
and start from there.  One simple way of doing this is to pretend that $f$ is a normal density 
with a mean and variance equal to the sample mean $\hat\mu$ and variance $\hat\sigma^2$ of the data, 
and then compute $R(f^{(r_0+2)})$ for this normal density.
To estimate the $r$th derivative $f^{(r)}$,
one can take the sample derivative of the KDE with a smooth kernel $k$, yielding
\begin{equation}
\label{eq:kde-derivative}
  \hat f_n^{(r)}(x) \approx \frac{1}{n h^{r+1}} \sum _{i=0}^{n-1} k^{(r)}\left( \frac{x - X_{i}}{h} \right).
\end{equation}
The asymptotically optimal $h$ to use in this KDE is 
\begin{equation}
\label{eq:hRecursion}
  h_{*}^{(r)} = \left( \frac{ (2r+1) \mu_0((k^{(r)})^2)}{\mu_2^2(k)^2 R(f^{(r+2)})n} \right)^{1/(2r+5)}.
\end{equation}
\hpierre{Florian: I assume this is is what you had in mind when you wrote ``density functional of order $2(r_0+2)$''
  in section 2 of your report on the sum of normals.  Is that it?   But it is important to realize that
	the functional will depend on $a$ and $b$.  Can we give a formula for this $R(f^{(r_0+2)})$ as a 
	function of $r_0$, $\hat\mu$, $\hat\sigma^2$, $a$, and $b$?}%
\hflorian{Pierre: Yes, exactly. Density functional would even be wrong to use now over the interval $[a,b]$. I do only have a formula for $r_0\in\{1,2\}$ for now, but I see what I can do.}%
%   ``reasonable'' $h$ to estimate $f^{(4)}$ (say) and start from there.
% This strategy can be used as well to estimate $R(f')$ for the histogram, using a KDE to estimate $f'$.
We will use this strategy to estimate a good $h$ in our experiments with MC and RQMC,
with a Gaussian kernel, with $r_0=2$.
%  $r_0=1$ for the histogram and $r_0=2$ for the KDE.

In this paper we always take $h$ to be the same for all $x\in[a,b]$.
It is possible to improve the kernel density estimation by using a locally varying bandwidth $h(x)>0$.
For instance, it is advantageous to have a larger $h = h(x)$ where $f(x)$ is smaller.
The interested reader is referred to
%all over the interval $[a,b]$, but we could often do better by varying $h$ over this interval
\cite{tSCO15a,tTER92a}.
\hart{Let's not introduce all the notation to describe local bandwidths only to not use it.
That will tire out the readers.}
% Note that the optimal $1/h$ is proportional to $n R(f'')$.   
% In a small subinterval $[c,c+\epsilon]\subset [a,b]$ in which $f''$ is almost constant,  
% the expected number of data points is $n_1 = n \int_{c}^{c+\epsilon} f(x)\d x$,
% so $1/h$ in this subinterval should be approximately proportional to 
% $n \int_{c}^{c+\epsilon} f(x) (f''(x))^2 \d x \approx n_1 (f''(c+\epsilon/2))^2$.
% That is, the value of $1/h$ used at some point $x$ should be approximately proportional
% to $f(x) (f''(x))^2$.  Varying $h$ adaptively is certainly interesting but it is outside
% the scope of this paper.  In the remainder, we assume that the
% same $h$ is used all over the interval.  This makes sense when the interval $[a,b]$
% does not contain thin tails or sub-regions in which $\max_x f(x) / \min_x f(x)$ is very large.

%%%%%%%%%%%%%%%%%%%%%%%%%%%%%%%%%%%%%%%%
\section{Error and variance bounds for RQMC integration}
\label{sec:rqmc-int}

We recall classical error and variance bounds for RQMC integration.
They can be found in \cite{rDIC10a}, \cite{vLEC09f}, \cite{rNIE92b}, and \cite{vOWE97b}, for example.
We will use them to obtain bounds on the AIV for the KDE.
%  kernel density estimators.stat

The integration error of $g : [0,1]^s\to\RR$ with 
the point set $P_n = \{\bu_1,\dots,\bu_{n}\}\subset [0,1]^s$ is
\[
  E_n = \frac{1}{n} \sum_{i=1}^{n} g(\bu_i) - \int_{[0,1]^s} g(\bu)\d\bu.
\]
Let $\frakv$ denote a subset of coordinates, $\frakv\subseteq \cS := \{1,\dots,s\}$.
%  and $-\frakv = \cS\setminus \frakv$ its complement. 
% For any function $g : [0,1)^s\to\RR$, 
For any $\bu = (u_1,\dots,u_s) \in[0,1]^s$ we denote by $\bu_{\frakv}$ the projection of $\bu$ on
the coordinates in $\frakv$ and by  $(\bu_\frakv,\bone)$ the point $\bu$ in which $u_j$ has been 
replaced by 1 for each $j\not\in \frakv$.  We write 
 $g_{\frakv} := \partial^{|\frakv|} g/\partial\bu_{\frakv}$ for the 
partial derivative of $g$ with respect to each of the coordinates in $\frakv$. % (when it exists).
When $g_{\frakv}$    
% $\partial^{|\frakv|} g/\partial\bu_{\frakv}$ 
exists and is continuous for $\frakv = \cS$, 
the \emph{Hardy-Krause variation} of $g$ is 
\begin{equation}
\label{eq:HK}
  V_{\rm HK}(g) = \sum_{\emptyset\not=\frakv\subseteq\cS} \int_{[0,1]^{|\frakv|}} 
	   \left|g_{\frakv}(\bu_{\frakv},\bone)\right| \d\bu_{\frakv}.
\end{equation}
\hflorian{Rewrote this to add the definitions of $\bu_{\frakv}$, $(\bu_{\frakv},\bone)$ and the right definition of $V_{\rm HK}$}%
The \emph{star-discrepancy} of $P_n$ is 
\[
  D^*(P_n) = \sup_{\bu\in[0,1]^s} 
    \left|\vol[\bzero,\bu) - \frac{|P_n\cap [\bzero,\bu)|}{n}\right|,
\]
where $\vol[\bzero,\bu)$ is the volume of the box $[\bzero,\bu)$.
The \emph{Koksma-Hlawka inequality} states that
\begin{equation}
  |E_n| \le V_{\rm HK}(g) \cdot D^*(P_n).              \label{eq:kokla}
\end{equation}
Several known construction methods give $P_n$ with  
$D^*(P_n) = \cO((\log n)^{s-1} /n) = \cO(n^{-1 +\epsilon})$ for all $\epsilon > 0$.
They include lattice rules and digital nets.
Therefore, if $V_{\rm HK}(g) < \infty$, it is possible to achieve $|E_n| = \cO(n^{-1+\epsilon})$
for the worst-case error.  
It is also known how to randomize the points of these constructions so that for 
the randomized points, $\EE[E_n] = 0$ and 
\begin{equation}
  \Var[E_n] = \EE[E_n^2] = \cO(n^{-2+\epsilon}).       \label{eq:var-bound}
\end{equation}
% (e.g., by a random shift modulo 1 for a lattice rule and by a digital random shift for a digital net).

\begin{comment}
With the (square) $L_2$-star discrepancy
\[
  D_2^2(P_n) = \int_{\bu\in[0,1]^s} 
    \biggl|\vol[\bzero,\bu) - \frac{|P_n\cap [\bzero,\bu)|}{n}\biggr|^2 \d\bu
\]
and the corresponding (square) variation
\begin{equation}
\label{eq:V2}
  V^2_{\rm 2}(g) = \sum_{\emptyset\not=\frakv\subseteq\cS} \int_{[0,1]^s} 
	   \biggl|\frac{\partial^{|\frakv|}g}{\partial\bu_{\frakv}}\biggr|^2 \d\bu,
\end{equation}
we obtain the (slightly different) inequality
\begin{equation*}
% \label{eq:kokla}
  |E_n| \le V_{2}(g) \cdot D_2(P_n).              
\end{equation*}
One always has $D_2(P_n) \le D^*(P_n)$, and therefore we also know how to construct 
points sets $P_n$ for which $D_2(P_n) = \cO(n^{-1+\epsilon})$.
%
Moreover, if $P_n$ is a digital net randomized by a \emph{nested uniform scramble} (NUS)
\cite{vOWE95a,vOWE97b} and $V_{2}(g) < \infty$, then $\EE[E_n]=0$ and 
$\Var[E_n] = \EE[E_n^2] = \cO(V_2^2(g) n^{-3} (\log n)^{s-1}) = \cO(V_2^2(g) n^{-3+\epsilon})$ for all $\epsilon > 0$.
\hpierre{Do we have a bound of the form $\Var[E_n] \le K V_2^2(g) n^{-3+\epsilon})$ here?
        This is what we need.}
\hpierre{We could also consider weighted discrepancies and variations with 
 projection-dependent weights, more general Holder inequalities, etc. 
 But I think not this time.  We want to keep it simple for now. }
\end{comment}

%%%%%%%%%%%%%%%%%%%%%%%%%%%%%%%%%%%%%%%%%%%%%%%%%%%%%%%%%%%%%%%
\section{Bounding the convergence rate of the AIV for a KDE with RQMC}
\label{sec:rqmc}
\label{sec:rqmc-iv}

%  \subsection{Asymptotic bounds on the IV}
%  \label{sec:rqmc-asymptotic}
\hpierre{Make sure all the results have the correct conditions.}

\hpierre{-- Recall that Bias is the same as for MC.}

\hpierre{-- Recall how the variance is bounded for MC: One bound the variance for $n=1$, using the 
   change of variable, then divide by $n$.
   Does not work for RQMC because points are not independent.   
	 But would be ok if they are negatively dependent! }

\hpierre{-- First, prove KH bound for $s=1$ and $g$ increasing.  
   Then state and prove the general result.   If it makes things simpler, we could perhaps prove
	 it first for the monotone case, then explains the generalization, or perhaps give the general proof
	 (non-monotone) only in the supplement, to make the paper easier to read.}

Replacing MC by RQMC does not affect the bias, because $\hat f_n(x)$ has the same expectation for both, 
but it can change the variance.  Before trying to bound the variance under RQMC, it is instructive to
recall how it is bounded under MC.  
To compute (or bound) the IV over an interval $[a,b]$, we compute (or bound)
$\Var[\hat f_n(x)]$ at an arbitrary point $x\in [a,b]$ and integrate this bound over $x$.
Since $\hat f_n(x)$ is an average of $n$ independent realizations of 
$Y(x) = k((x - X)/h) / h$, it suffices to compute $\Var[Y(x)]$ and we have $\Var[\hat f_n(x)] = \Var[Y(x)]/n$.
With the change of variable $w = (x-v)/h$, we obtain \cite[page 143]{tSCO15a}:
\begin{eqnarray*}
  \Var[Y(x)] &=& \frac{1}{h^2} \int_{-\infty}^\infty  k^2\left(\frac{x-v}{h}\right) f(v)\d v  - \EE^2[Y(x)] \\
	        &=& \frac{1}{h} \int_{-\infty}^\infty k^2(w) f(x-hw)\d w  - \EE^2[Y(x)]
	        = \frac{f(x)}{h} \int_{-\infty}^\infty k^2(w) \d w \, - f^2(x) + \cO(h).
\end{eqnarray*}
Integrating over $x\in [a,b]$,
%  since $f(x)$ integrates to 1,
gives $\IV = p_0 \,\mu_0(k^2)/(nh) - R(f)/n + \cO(h/n)$ where
$p_0 = \int_a^b f(x) \d x \le 1$.  %   is usually close to 1.

With RQMC, this also holds for a single RQMC point $\bU_i$ and $X = X_i = g(\bU_i)$,
but to obtain $\Var[\hat f_n(x)]$, we can no longer just divide $\Var[Y(x)]$ by $n$,
because the $n$ realizations of $Y(x)$ are not independent.
RQMC is effective if and only if these realizations are negatively correlated, 
in the sense that if $Y_i = h^{-1} k((x - X_i)/h)$,
then $\sum_{i\not=j} \Cov(Y_i, Y_j) \le 0$.
This would imply that RQMC can never be worse than MC, but this seems hard to prove.

We now take a different path, in which we examine how the KH inequality \cref{eq:kokla} 
can be used to bound $\Var[\hat f_n(x)]$.
%  the AIV for a KDE with RQMC.
With $X = g(\bU)$, we can write
\begin{equation}
 \hat f_n(x)  % = \frac{1}{nh} \sum_{i=1}^n  k\left(\frac{x-g(\bU_i)}{h}\right)    
        = \frac{1}{n} \sum_{i=1}^n  \tilde g(x,\bU_i)   
 \quad\mbox{ where }\quad
 {\tilde g(x,\bU_i)} := Y_i = \frac{1}{h} k\left(\frac{x-g(\bU_i)}{h}\right).
									\label{eq:kde2}
\end{equation}
%  where $\tilde g(x,\bU_i) := h^{-1} k((x-g(\bU_i))/{h})$.
Thus, $\hat f_n(x)$ can be interpreted as an RQMC estimator of 
$\EE[\tilde g(x,\bU)] = \int_{[0,1]^s} \tilde g(x,\bu)\d\bu$.
To apply the bound in \cref{eq:var-bound} to this estimator,
we need to bound the variation of $\tilde g(x,\cdot)$, by bounding each term of the sum in \cref{eq:HK}.

To provide insight, we first examine the special case where $s=1$ and $g$ is nondecreasing over $[0,1]$,
under \Cref{ass:kGeneral}.
Then $\tilde g(x,u) = k((x-g(u))/h)/h$ is nonincreasing over the $u$
with $g(u)\le x$ and nondecreasing over $u$ with $g(u)\ge x$. In that case
$V_{\rm HK}(\tilde g(x,\cdot))$ is the ordinary one-dimensional total variation, and then
%\pierre{Yes. This is equivalent to the change of variable, and even more intuitive.}
\begin{equation}
  V_{\rm HK}(\tilde g(x,\cdot))  \le 
  \left|\frac1hk(0)-\frac1hk\left(\frac{x-g(0)}{h}\right)\right|
 +\left|\frac1hk(0)-\frac1hk\left(\frac{x-g(1)}{h}\right)\right|
 \le \frac{2k(0)}h.
 \label{eq:VHK1}
\end{equation}
%  which has one less power of $h^{-1}$ than the bound obtained via \cref{eq:supnormbound}.
The same bound holds for nonincreasing functions $g$. 
More generally, if $g$ is monotone within each of $M$ intervals that partition the domain $[0,1]$
then  $ V_{\rm HK}(\tilde g(x,\cdot))  \le 2Mk(0)/h$. 
\hpierre{No additional condition needed?}%
The factor of $2$ is necessary because
$k$ is potentially increasing and then decreasing within each of those intervals.
Note that there are one-dimensional point sets $P_n$ with $D^*(P_n) \le 1/n$.
With such point sets, we obtain  
$\Var[\hat f_n(x)] \le (2M k(0))^2 / (nh)^2$, so $\AIV = \cO((nh)^{-2})$.
With $h = \Theta(n^{-1/3})$, this gives $\AMISE = \cO(n^{-4/3})$.

We now consider the general case $s \ge 1$. To bound $V_{\rm HK}(\tilde g(x,\cdot))$ we will make
a similar change of variables as for the IV under MC. 
We need additional assumptions on $k$ and $g$.

\begin{assumption}
\label{ass:kernel}
The kernel function $k:\RR\to\RR$ is $s$ times continuously differentiable and its derivatives 
up to order $s$ are integrable and uniformly bounded over $\RR$.
\hflorian{Added the integrability assumption}
\end{assumption}

\begin{assumption}
\label{ass:g}
Let $g:[0,1]^{s}\to\RR$ be piecewise monotone in each coordinate $u_j$ when the other coordinates are fixed,
with a number of monotone pieces (which is 1 plus the number of times that the function switches 
from strictly decreasing to strictly increasing or vice-versa, in $u_j$)
that is bounded uniformly in $\bu$ by an integer $M_j$.
%  , and let $M = \min_{1\le j\le s} M_j$.   
We also assume that the first-order partial derivatives of $g$ are continuous and that
% $\|g_{\frak{w_1}} g_{\frak{w_2}}\cdots g_{\frak{w_\ell}}\|_{\infty} < \infty$
% for all selections of non-empty, mutually disjoint index sets 
% $\frak{w_1},\dots,\frak{w_\ell}\subseteq\cS$.
$\|g_{\frakv}\|_{\infty}<\infty$ for all $\frakv\subseteq\cS$.
This implies that any product of partial derivatives of $g$ of order at most one in each variable is integrable.
\hflorian{Changed the norm type in this assumption from $L_1$ to $L_{\infty}$. Now, we also do not need to bound the norm of the product of derivatives, as this is already guaranteed now.}%
\end{assumption}
\hflorian{Actually we need the integrability condition only for collections of disjoint index sets $\mathfrak w_1, \dots,\mathfrak w_\ell$ such that they either partition $\cS$ or $|\bigcup_{l=1}^{\ell}\mathfrak{w}_l| = s-1 $.  Should we keep it as is for the sake of readability, or is it worth stating the weaker condition?}
\hflorian{The proof of \cref{prop:vhk} (and 
  \cref{cor:miseHK}) requires slightly weaker conditions on $g$ than those 
  of \cref{ass:g}. We need the integrability condition \cref{eq:gIntegrability}
only for collections of disjoint index sets $\mathfrak w_1,\dots,\mathfrak w_\ell$ such
that they either partition $\cS$ or $|\bigcup_{l=1}^{\ell}\mathfrak{w}_l| = s-1 $.}
\hflorian{I took this condition as it 1) is weaker and 2) comes in more naturally than what we had before, i.e. $|g_{\frak{w}}|\leq B_{\frak{w}}<\infty$. In fact, we had $|g_{\{j\}}|\leq B_{j}<\infty$, but the proof shows that we need it for mixed derivatives too.}
\hflorian{This will be used in the definition of $c_j$ below as well as in the proof of Proposition~\ref{prop:vhk}. I thought it might by a good idea to put this here\ldots}

Because the Hardy--Krause variation \cref{eq:HK} involves mixed partial derivatives of $\tilde g(x,\cdot)$ 
of order up to $s$, things unfortunately become considerably more complicated than for MC. Roughly 
speaking, every derivative causes an additional factor $h^{-1}$, while we may dispose of 
only one such factor through a change of variables. This is reflected in \cref{prop:vhk} below.
Similar to the one-dimensional case, we need to take into account how often $g$ changes 
its monotonicity direction, and this is captured by the $M_j$'s.
%To this end, we define  $\ind(g,j,\bu)$ as 1 plus  the number 
%of times that $g(\bu)$  switches between being strictly decreasing and strictly increasing
%in $u_j$ when the other coordinates are fixed.
%We call $\ind(g,j) := \|\ind(g,j,\cdot)\|_{\infty}$ the \emph{index of $g$ w.r.t. $j$}.
%
For each $j\in\cS$, let 
%  $G_j = \left\| \prod_{\ell\in \cS\setminus \{j\}} g_{\{\ell\}}\right\|_{\infty}$, 
\begin{eqnarray*}
  G_j &=& \left\| \prod_{\ell\in \cS\setminus \{j\}} g_{\{\ell\}}\right\|_{\infty}, \nonumber\\ 
% \label{eq:cj} \quad
  c_j &=& M_{j} \cdot\left(\|k^{(s)}\|_1 \cdot G_j + \II(s=2) \cdot  \|k^{(s)}\|_{\infty} \cdot \|g_{\{1,2\}}\|_{1}\right) < \infty,
%	\quad \mbox{ and } c = \min_{j\in\cS} c_j,
\end{eqnarray*}
where $\II(\cdot)$ is the indicator function, so the expression for $c_j$ contains an extra term when $s=2$.
The source of this extra term is that for $s=2$, the only partition of $\{1,2\}$ 
which contains no singletons is $\{1,2\}$ itself, and it gives a term in $h^{-2} = h^{-s}$,
whereas for $s > 2$, all the extra terms are $\cO(h^{-s+1})$.
%   ; see the proof of the proposition.}
%  This is a technical consequence for which the root cause is that the only partition of {1,2} 
%  which contains no singletons is {1,2} itself, unlike the case of {1,2,...,s} with s>2.
Our main result of this section is:

\begin{proposition}
\label{prop:vhk}
 Let $k$ and $g$ satisfy \Cref{ass:kGeneral,ass:kernel,ass:g}, and $c = \min_{j\in\cS} c_j$.
% For every $j_0\in\cS$
Then the Hardy-Krause variation of % the function
 $\tilde{g}(x,\bu) = h^{-1}k((x-g(\bu))h^{-1})$ (as a function of $\bu$) satisfies
\begin{equation}
\label{eq:bound-vhk}
 V_{\rm HK}(\tilde{g}(x,\cdot)) \leq  c h^{-s} + \cO(h^{-s+1}).
\end{equation}
\hflorian{I hate this exception for $s=2$, especially because the resulting IV bound is only better than with MC for $s\leq3$. On the other hand, it only affects the constant. A workaround is not really obvious to me, since the mixed derivative $g_{\{1,2\}}$ wouldn't cancel after the change of variables.}
\hflorian{The need for the kernel to be the gaussian might be obsolete, see further below.}
\hflorian{I would like to include the implied constant in this proposition but in its current form it is too tedious to explain it right here. For now, I have put the constant in the separate Remark~\ref{rem:vhkConstant} after the proof. If we do not need the subdivision of $\RR$ w.r.t. the zeros of $k^{(s)}$ (see below), it will be much easier.}
\end{proposition}

Note that the constants $c_j$ depend on $g$ via $M_j$ and $G_j$ (plus an extra term when $s=2$).
A large $M_j$ means that $g$ changes the sign of its slope many times in the direction of coordinate $j$.
A large $G_j$ means that the product of the slopes in the directions of the other coordinates 
can attain large values.
% has a large absolute value on average.  
\hflorian{Changed this one from ``has a large absolute value on average'' to ``can attain large values''.}%
When we have both, then $c_j$ is large. 
Observe that the bound \cref{eq:bound-vhk} uses the \emph{smallest} $c_j$.
In case $g$ is constant with respect to one coordinate $\ell\not=j$, then $G_j=0$, $c_j=0$, and $c=0$.
That is, the term in $h^{-s}$ disappears from \cref{eq:bound-vhk} and the bound becomes $\cO(h^{-s+1})$.
This agrees with the fact that $g$ is then effectively a $(s-1)$-dimensional function. 
Likewise, if $g$ is almost constant (has very little variation) with respect 
to one or more coordinate(s) $\ell\not=j$, then $G_j$ and therefore $c_j$ will be small, 
unless the other terms in the product are very large.
Before proving this proposition, we state a corollary that bounds the AIV and the AMISE rates under RQMC.

\begin{corollary}
	\label{cor:miseHK}
  Let $k$ and $g$ satisfy \Cref{ass:kGeneral,ass:f,ass:kernel,ass:g}. For a KDE with 
  kernel $k$, with the underlying observations obtained via sets $P_n$ of $n$ RQMC points for which
  $D^{*}(P_n)=\cO(n^{-1+\epsilon})$ for all $\epsilon > 0$ when $n\to\infty$, 
  by combining \cref{eq:bound-vhk} with \cref{eq:kokla} and squaring, we find that 
  $\AIV = \cO(n^{-2+{\epsilon}} h^{-2s})$ for all ${\epsilon}>0$.  
	%  \mpierre{$\tilde\epsilon = ???$   Not for all!}  
	Then, by taking $h = \Theta(n^{-1/(2+s)})$, \hpierre{Removed the $+\epsilon/4$, not needed.}% 
  we obtain that $\AMISE = \cO(n^{-4/(2+s) + \epsilon})$ for all $\epsilon>0$. 
	The exponent of $n$ in this AMISE bound beats the MC rate for $s<3$ and is almost equal to the MC rate for $s=3$.
\hflorian{We need to clearly state a unified assumption for $k$ and $f$ somewhere. $k$ is always implicitly assumed to be a pdf with $0<\mu_2(k)<\infty$ and $f$ should fulfill $R(f'')<\infty$. The last two conditions are necessary for $\AISB = \cO(h^{4})$.}
\end{corollary}
% \begin{corollary}
%  Let $g$ satisfy Assumption~\ref{ass:g} and let $k$ be a symmetric probability density 
%  function with $0<\mu_2(k)<\infty$ which is subject to Assumption~\ref{ass:kernel}.
% %  
%  \florian{Maybe a global assumption for $k$ would be good somewhere in the beginning}
% %  
% For a  KDE with kernel $k$ for a density $f$ with $R(f'')<\infty$ whose underlying 
% observations are obtained by an RQMC point set $P_n$ with $D_n^{*}=n^{-1+\tilde\epsilon}$, we have that $\AIV = \cO(n^{-2+\tilde{\epsilon}}h^{-2s})$ for all $\tilde{\epsilon}>0$. Consequently, with $h = \Theta(n^{-1/(2+s) + \epsilon/4})$ we obtain $\AMISE = \cO(n^{-4/(2+s) + \epsilon})$ for all $\epsilon>0$.
% \end{corollary}

Let $\Pi(\frakv)$ denote the set of all partitions of a set of coordinate indices $\frakv\subseteq\cS$,
and let $\Pi_1(\frakv)$ denote the subset of all partitions that contain at least one singleton.
For each partition $P \in \Pi_1(\frakv)$, we select a particular singleton and denote it by $\{j(P)\}$. 
Removing that singleton from $P$ yields a partition of $\frakv^* = \frakv\setminus\{j(P)\}$ which we denote by $P^*$. 
\hflorian{Included the definition of $\frakv^*$.}
\hflorian{Unfortunately, I do not see a way to completely dispose of the symbol $j_P$, as we still need it for $\ind(g,j_P)$ too. Simply taking $\max_j\ind(g,j)$ does not work, as the $L_1$ norm depends on $j$ too. 
Everything else that I have tried so far results in a pile of additional notation. I'll keep thinking about it...}
% In addition, let 
% $$\omega_{P,g}=\max_{\{j_P\}\in P} \ind(g,j_P) \Big\|\prod_{\mathfrak{w}\in P_*}g_{\mathfrak w}\Big\|_{1}.$$

The proof of the proposition will use the following lemma, which describes when 
the aforementioned change of variable works and how it removes a factor $1/h$ from the bound.

\begin{lemma}
  \label{lem:cov}
  Let \Cref{ass:kGeneral,ass:kernel,ass:g} hold, let $h>0$, $\frakv\subseteq\cS$, 
	and  $P\in\Pi_1(\frakv)$. Then 
 \begin{eqnarray*}
    \int_{[0,1]^{|\frakv|}} \left| k^{(|P|)}\left( \frac{x-g(\bu_{\frakv},\bone)}{h} \right) 
		    \cdot \prod_{\mathfrak w\in P} g_{\mathfrak w}(\bu_{\frakv},\bone)\right|\d\bu_{\frakv} 
		&~\leq~&  h \cdot M_{j(P)} \cdot\left\|\prod_{\mathfrak{w}\in P^*} g_{\mathfrak w}\right\|_{\infty}
		    \cdot \|k^{(|P|)}\|_{1}.
 \end{eqnarray*}
%    with 
%    \begin{equation*}
%     c_{g,k,P,j_P} = \ind(g,j_P)\Big\|\prod_{\mathfrak{w}\in P_*}g_{\mathfrak w}\Big\|_{p}~\|k^{(|P|)}\|_{\infty}.
%    \end{equation*}
\end{lemma}

\begin{proof}
We assume without loss of generality that $1\in \frakv$ and $j(P)=1$. 
We make the change of variables
\begin{equation}
 \label{eqn:cov}
 u_{1}\mapsto w = (x-g(\bu_{\frakv},\bone))/{h}.
\end{equation}
For any $\bu_{\frakv}\in[0,1]^{|\frakv|}$ with fixed $\bu_{\frakv^*}\in[0,1]^{|\frakv|-1}$, we partition $[0,1]$ into a part 
$\mathcal N(\bu_{\frakv^*})$ where $g(\bu_{\frakv},\bone)$ is constant in $u_{1}$ and into sets 
$\mathcal D_l(\bu_{\frakv^*})$, $1\leq l \leq L(\bu_{\frakv^*})\leq M_1$, on which it is either strictly
decreasing or strictly increasing in $u_{1}$. 
Since $g_{\{1\}}$ is continuous by assumption, each 
of these sets is measurable. The restriction of $g_{\{1\}}$ to $\mathcal{N}(\bu_{\frakv^*})$ equals 0 identically.
In all the other sets $\mathcal{D}_l(\bu_{\frakv^*})$ we apply the change of variables  
%  $\varphi_{l}:\mathcal D_{l}(\bu_*)\to\RR$, $\Phi_{l}(u_1) = w$ with  $w$ as in 
\cref{eqn:cov}.
% Note that $\Phi_l$ is injective and its derivative is smooth and invertible. Moreover,  $|\det^{-1} (\nabla\Phi_{l})|= h/|g_{\{j_P\}}(\bu)|$. 
Considering this in the left-hand side of the claim leads to
\begin{align*}
     \int_{[0,1]^{|\frakv|}}  & \left| k^{(|P|)} 
	    \left( \frac{x-g(\bu_{\frakv},\bone)}{h} \right) 
	   \prod_{\mathfrak w\in P} g_{\mathfrak w}(\bu_{\frakv},\bone)\right|\d\bu_{\frakv} \\
   & = \
  \int_{[0,1]^{|\frakv|-1}} \sum_{l = 1}^{L(\bu_{\frakv^*})} \int_{\mathcal{D}_l(\bu_{\frakv^*})}
	 \left|k^{(|P|)}\left( \frac{x-g(\bu_{\frakv},\bone)}{h} \right) \prod_{\mathfrak w\in P} 
	       g_{\mathfrak w}(\bu_{\frakv},\bone)\right| \d u_1 \d\bu_{\frakv^*} \\
   & \leq \ h
    \int_{[0,1]^{|\frakv|-1}} 
    L(\bu_{\frakv^*}) \int_{-\infty}^{\infty} \left|k^{(|P|)}(w) 
		   \prod_{\mathfrak w\in P^*}g_{\mathfrak w}(\bu_{\frakv},\bone)\right| \d w  \d \bu_{\frakv^*} \\  
			  %  \prod_{j\in\cS\setminus \{1\}}\d u_j \\
   & \leq \ h \cdot M_1 \cdot \|k^{(|P|)}\|_{1} \cdot  
	   \left\| \prod_{\mathfrak w\in P^*} g_{\mathfrak w}\right\|_{\infty},
\end{align*}
 where we used H{\"o}lder's inequality in the last step.
\end{proof}

%%%%%%%%%%%%%%%%%%%%%%%%%%%%%%%%%%%%%%%%
\begin{proof}[Proof of \cref{prop:vhk}]
 We rewrite each summand w.r.t. $\frakv\subseteq\cS$ in \cref{eq:HK} with the help of Fa{\`a} di 
 Bruno's formula (see \cite[Proposition 1]{mHAR06a}) as follows
{\small\begin{align}
 \int_{[0,1]^{|\frakv|}}\left|
%  \frac{\partial^{|\frakv|}\tilde{g}(\bu_{\frakv},\bone)}{\partial\bu_{\frakv}}
 \tilde{g}_{\frakv}(x,\bu_{\frakv},\bone)
 \right|\d\bu_{\frakv}
 &=  \small
 \frac{1}{h}\int_{[0,1]^{|\frakv|}}\Bigg|  \sum_{P\in\Pi( \frakv)} k^{(|P|)}\left( \frac{x-g(\bu_{\frakv},\bone)}{h} \right) \prod_{\mathfrak{w}\in P}\frac{\partial^{|\mathfrak{w}|}}{\partial\bu_{\mathfrak{w}}}\left( \frac{x-g(\bu_{\frakv},\bone)}{h} \right)\Bigg| \d\bu_{\frakv}
 \nonumber\\
 &\leq 
 \sum_{P\in\Pi( \frakv)} \frac{1}{h^{|P|+1}} \int_{[0,1]^{|\frakv|}}\Bigg| k^{(|P|)}\left( \frac{x-g(\bu_{\frakv},\bone)}{h} \right) \prod_{\mathfrak{w}\in P} g_{\frak{w}}(\bu_{\frakv},\bone)\Bigg|\d\bu_{\frakv}.
\label{eqn:faaDiBruno}
\end{align}}%
If $P\in\Pi_1(\frakv)$,  % contains at least one singleton $\{j(P)\}$, 
we bound the corresponding summand in \cref{eqn:faaDiBruno} via \cref{lem:cov}. 
If $P\not\in\Pi_1(\frakv)$, we apply H{\"o}lder's inequality to obtain the upper bound
\begin{equation*}
% \label{eqn:partitionNotSingletons}
  \frac{1}{h^{|P|+1}} \cdot\|k^{(|P|)}\|_{\infty} \cdot \Big\|\prod_{\mathfrak{w}\in P} g_{\frak{w}} \Big\|_{1}.
\end{equation*}
Furthermore, we observe that each element of $P\in\Pi(\frakv)\setminus\Pi_1(\frakv)$ has a cardinality of at least 2. Therefore, $P$ can contain at most $\lfloor|\frakv|/2\rfloor$ elements.
\hflorian{Included this sentence following Art's request.}%
This gives the following bound on $V_{\rm HK}(\tilde g(x,\cdot))$, which holds for any $j\in\cS$:
{\small\begin{eqnarray*}
 V_{\rm HK}(\tilde g(x,\cdot))  
 &\leq&
 \sum_{\emptyset\neq \frakv \subseteq\cS}
%  \left[ \sum_{ P\in\Pi_1(\frakv) }  \frac{M_{j(P)} C_{j(P)}}{h^{|P|}}\|k^{(|P|)}\|_1 +
 \left[ \sum_{ P\in\Pi_1(\frakv) } \hskip-5pt \frac{M_{j(P)}}{h^{|P|}}\|k^{(|P|)}\|_1 \cdot \Big\|\prod_{\mathfrak{w}\in P^*} g_{\frak{w}} \Big\|_{\infty}+
 \hskip-10pt\sum_{P\in\Pi(\frakv)\setminus\Pi_1(\frakv) } 
 \hskip-5pt \frac{1}{h^{|P|+1}} {\|k^{(|P|)}\|_{\infty} } \cdot \Big\|\prod_{\mathfrak{w}\in P} g_{\frak{w}} \Big\|_{1}\right] \\
 &\leq& \ h^{-s} M_j G_j ~\|k^{(s)}\|_{\infty}+ \cO(h^{-s+1}) % \\[5pt]
  \ + \sum_{\emptyset\neq \frakv \subseteq\cS} h^{-\lfloor|\frakv|/2\rfloor-1} 
	\hskip-10pt \sum_{P\in\Pi(\frakv)\setminus\Pi_1(\frakv) }  \|k^{(|P|)}\|_{\infty} 
	\cdot \Big\|\prod_{\mathfrak{w}\in P} g_{\frak{w}} \Big\|_{1}.
\end{eqnarray*}}%
For $s=1$ this already proves the claim. 
For $s=2$, the only partition of $\cS$ that
contains no singleton is $\cS$ itself and the result follows. 
For $s\geq3$ we have $\lfloor s/2\rfloor+1\leq s-1$, and then
% \begin{equation*}
\[  \sum_{\emptyset\neq \frakv \subseteq\cS} h^{-\lfloor|\frakv|/2\rfloor-1} \sum_{P\in\Pi(\frakv)\setminus\Pi_1(\frakv)}  \|k^{(|P|)}\|_{\infty} \cdot \Big\|\prod_{\mathfrak{w}\in P} g_{\frak{w}} \Big\|_{1} = \cO(h^{-s+1}).
%   \qedhere
\]
% \end{equation*}
\end{proof}

The bound of \cref{prop:vhk} suggests that the IV could converge at a much worse rate 
with RQMC than with MC when $s$ is large. 
However, the next proposition, based on a result of \cite{vOWE98b}, provides a different bound
that does not grow as $h^{-2s}$ when $h$ decreases, for a particular type of RQMC point set,
namely a $(t,m,s)$-net in base $2$ randomized by a nested uniform scramble (NUS).
This type of point set contains $2^m$ points in $s$ dimensions, the $t$ parameter measures the uniformity
in some sense (the smaller the better) \cite{rDIC10a,rNIE92b}, 
and the NUS shuffles the points in some particular way \cite{vOWE95a,vOWE97b}.
% \cite{vOWE95a,vOWE97b} and $V_{2}(g) < \infty$, then $\EE[E_n]=0$ and 
% $\Var[E_n] = \EE[E_n^2] = \cO(V_2^2(g) n^{-3} (\log n)^{s-1}) = \cO(V_2^2(g) n^{-3+\epsilon})$ for all $\epsilon > 0$.
\hpierre{Here we use $b$ to denote both the base of a digital net and the right boundary of the interval 
 $[a,b]$, because these notations are standard, and it is easy to distinguish which is which by the context.}
We state the following result for base $b=2$, but it can be extended to a general prime base $b \ge 2$.

\begin{proposition}
\label{prop:owen98}
Let $P_n$ be a $(t,m,s)$-net in base $2$ randomized by NUS, and let \cref{ass:kGeneral} hold.
%  assume that $\int_{[0,1)^s} \tilde g(\bu) \d\bu < \infty$.
Then the IV of $\hat{f}_n$ satisfies
\[
  \IV \leq 2^t 3^s \mu_0(k^2)/(nh) + \cO(h/n). 
\]
Moreover for any fixed $s\ge 1$, there is a fixed $t \ge 0$ for which we know how to 
construct a $(t,m,s)$-net $P_n$ in base $2$ for any integer $m \ge 1$. 
By using such a sequence of nets with NUS, we get $\IV = \cO(1/(nh))$,
and then by taking $h = \Theta(n^{-1/5})$, we obtain $\MISE = \cO(n^{-4/5})$.
That is, the MISE never converges at a worse rate than with plain MC.
\end{proposition}

\begin{proof}
Let $\Var_{\rm MC}$ and $\Var_{\rm NUS}$ denote the variance under MC and under NUS, respectively.
Likewise, for any given pair $(n,h)$, let $\IV_{\rm MC}(n,h)$ and $\IV_{\rm NUS}(n,h)$ denote the IV 
under MC and under NUS, respectively, and similarly for the MISE.
Under \cref{ass:kGeneral}, $\tilde g(x,\cdot)$ is square-integrable over $[0,1]^s$ for any $x\in[a,b]$,
so we can apply Theorem 1 of \cite{vOWE98b}, which tells us that 
\hpierre{The constant $b^t$ in front might need additional factors for $b\not=2$...}
\[
  \Var_{\rm NUS}[\tilde g(x,\bU)] ~\le~ 2^t 3^s  \Var_{\rm MC}[\tilde g(x,\bU)].
\]
By integrating, we obtain $\IV_{\rm NUS}(n,h) \le 2^t 3^s  \IV_{\rm MC}(n,h)$.
We saw earlier that 
%  $\Var[\tilde g(x,\bU)] = f(x) \mu_0(k^2)/h + \cO(h)$, and therefore
$\IV_{\rm MC}(n,h) \le \mu_0(k^2)/(nh) - R(f)/n + \cO(h/n)$, 
and this proves the displayed inequality.

For the second part, for any $s\ge 1$, there is a fixed $t \ge 0$ 
for which we know how to construct a $(t,s)$-sequence in base $2$;
see \cite{rSOB67a} and \cite[Section 4.5]{rNIE92b}, for example.
For any integer $m$, the first $2^m$ points of such a sequence form a $(t,m,s)$-net in base 2.
%  In practice, one may obtain slightly better nets by constructing a different net for each $m$,
%  but this does not improve the convergence rate. 
Thus, $t$ can be assumed to be bounded uniformly in $m$, and therefore
\hpierre{Florian: can you check for a specific theorem that we can cite for this, with page number?}%
%   This means that for fixed $s$, one has 
$\IV_{\rm NUS}(n,h) = \cO(\IV_{\rm MC}(n,h)) = \cO(1/(nh))$.
Since MC and NUS give the same ISB, we also have 
$\MISE_{\rm NUS}(n,h) \le 2^t 3^s  \MISE_{\rm MC}(n,h) = \cO(1/(nh) + h^4) = \cO(h^{-4/5})$ 
if we take $h = \Theta(n^{-1/5})$. 
\end{proof}

%%%%%%%%%%%%%%%%%%%%%%%%%%%%%%%%%%%%%%%%%%%%%%%%%%%%%%%%%%%%%%%%%%%%%%%
\section{Stratified sampling of $[0,1)^s$}
\label{sec:strat}

In this section, we examine how plain stratified sampling of the unit hypercube 
can reduce the IV of the KDE compared with MC.
We consider point sets $P_n$ constructed as in \Cref{ass:strat} below.
This type of stratified sampling can never increase the IV compared with MC.  
We prove this via a standard variance decomposition argument.
Then, under the additional condition that $g(\bu)$ is monotone with respect to each coordinate of $\bu$,
we prove an IV bound that converges at a faster rate than the IV under MC when $s < 5$.
The KH inequality and the variation of $g$ are not involved in the IV bound developed here;
we work directly with the variance.
For this reason, the bound will not contain the annoying factor $h^{-2s}$ as in \cref{prop:vhk}.
On the other hand, the exponent of $n$ will not be as good.
Our main results are \Cref{prop:stratification0,prop:stratification}, and \Cref{cor:miseStrat}.

\begin{assumption}
\label{ass:strat}
The hypercube $[0,1)^s$ is partitioned into $n = q^{s}$ congruent cubic cells 
$S_{\bi} := \prod_{j=1}^{s}\left[{i_j}/{q},{(i_j+1)}/{q} \right)$, 
$\bi \in \bI = \{\bi = (i_1,i_2,\dots,i_s) : 0\leq i_j<q$ for each $j\}$, for some integer $q\geq 2$.
We construct $P_n = \{\bU_1,\dots,\bU_n\}$ by sampling one point uniformly in each subcube $S_{\bi}$,
independently across the subcubes, and put $X_i = g(\bU_i)$ for $i=1,\dots,n$. 
\end{assumption}

\begin{proposition}
\label{prop:stratification0}
Under \Cref{ass:kGeneral,ass:strat},
%   $\hat{f}_n$ be a KDE with kernel $k$ obtained from $X_1,X_2\dots,X_{n}$.
the IV of a KDE $\hat{f}_n$ with kernel $k$ never exceeds the IV of the same estimator under standard MC,
which satisfies $\IV \le \mu_0(k^2) /(nh) - R(f)/n + \cO(h/n)$.
\end{proposition}

\begin{proof}
We can decompose the variance under MC as
\begin{eqnarray*}
 \Var[\tilde g(x,\bU)] &=& \EE[\Var[\tilde g(x,\bU) \mid \bU\in S_{\bi}] + \Var[\EE[\tilde g(x,\bU) \mid \bU\in S_{\bi}] \\
     &=& \frac1n \sum_{\bi \in \bI} \Var[\tilde g(x,\bU) \mid \bU\in S_{\bi}] + \frac1n \sum_{\bi \in \bI} {(\mu_{\bi}-\mu)^2},
\end{eqnarray*}
where $\mu = \EE[\tilde g(x,\bU)]$ and $\mu_{\bi} = \EE[\tilde g(x,\bU) \mid \bU\in S_{\bi}]$.
By sampling exactly one point in each cell $S_{\bi}$, the stratified sampling removes the second term,
and the first term remains the same.  Therefore, stratification never increases $\Var[\hat{f}_n(x)]$.
The second term also indicates how the amount of variance reduction depends on how the $\mu_{\bi}$ vary between boxes.
\end{proof}

Now we know that stratification can do no harm.  To show that it can also improve the convergence rate of the MISE,
we need additional conditions. 

\begin{assumption}
\label{ass:gmonotone}
For each $j\in\cS$, the function $g : [0,1)^s\to \RR$ is \emph{monotone} in $u_j$ when the other $s-1$ 
coordinates are fixed, and the direction of monotonicity in $u_j$ (nondecreasing or nonincreasing) 
is the same for all values of the other coordinates.  
Without loss of generality, we will assume in the rest of the paper that it is nondecreasing 
in each coordinate.  (If it is nonincreasing in $u_j$, one can simply replace $u_j$ by $1-u_j$
in the definition of $g$ and this does not change the distribution of $X = g(\bU)$.)
\end{assumption}

\begin{proposition}
 \label{prop:stratification}
Let \Cref{ass:kGeneral,ass:strat,ass:gmonotone} hold and let
$\hat{f}_n$ be a KDE with kernel $k$ obtained from $X_1,X_2, \dots,X_{n}$.
%  If $g$ is monotone in each coordinate and $k'\in L_1(\RR)$, then
Then the IV of $\hat{f}_n$ satisfies
\[
 \IV \leq  {(b -a)}s \cdot k^2(0) \cdot h^{-2}  n^{-(s+1)/s}.
\]
\end{proposition}

\begin{corollary}
\label{cor:miseStrat}
 Under \cref{ass:kGeneral,ass:f,ass:strat,ass:gmonotone}, 
 the AMISE bound is minimized by taking $h = \kappa n^{-(s+1)/(6s)}$ with 
 $\kappa^6 = [{(b -a)}s \cdot k^2(0)] / [(\mu_2(k))^2 R(f'')/2]$, and this gives
 $\AMISE = K n^{-\nu}$ with $\nu = (2/3)(s+1)/s$ and 
 $K = {(b -a)}s \cdot k^2(0) \kappa^{-2} + (\mu_2(k))^2 R(f'') \kappa^4 / 4$.
%   by taking $h = \Theta(n^{-(s+1)/(6s)})$,  the KDE estimator has $\AMISE = \cO(n^{-(2/3)(s+1)/s})$. 
 The exponent of $n$ in this bound beats the MC rate for all $s<5$ and is equal to the MC rate for $s=5$.
\end{corollary}

The corollary is straightforward to prove.  It suffices to minimize with respect to $h$
the sum of the AISB (given in \Cref{sec:density-estimators}) and the IV bound.
\hpierre{In fact, the AMISE bound is minimized by taking $h = \kappa n^{-(s+1)/(6s)}$ with 
 $\kappa^6 = [{(b -a)}s \cdot k^2(0)] / [(\mu_2(k))^2 R(f'')/2]$, and this gives
$\AMISE = K n^{-\nu}$ with $\nu = (2/3)(s+1)/s$ and 
$K = {(b -a)}s \cdot k^2(0) \kappa^{-2} + (\mu_2(k))^2 R(f'') \kappa^4 / 4$.}

Our proof of \cref{prop:stratification} is inspired by~\cite[Proposition 6]{vLEC08a}.
It combines three lemmas that we state and prove before proving the proposition. 
These lemmas could be useful in other contexts as well.
The first lemma bounds the variance of the KDE in terms of a uniform bound $K_{n}$
on the variance of the empirical cdf estimator.   The other lemmas bound $K_n$.

Let $F$ be the cdf of the random variable $X$.
For any $x\in\RR$, the empirical cdf of a sample $X_1,X_2\dots,X_{n}$ evaluated at $x$ is
$\hat{F}_n(x) := n^{-1}\sum_{i=1}^{n} \II[X_i \leq x]$.  
We denote the difference by
% \begin{equation*}
$\Delta_n(x) = \hat{F}_n(x) - F(x)$.
% \end{equation*}
Let $K_n := \sup_{x\in\RR} \var[\Delta_n(x)] = \sup_{x\in\RR} \var[\hat{F}_n(x)]$.

\begin{lemma}
 \label{lem:stratK}
Under \cref{ass:kGeneral}, for all $x\in\RR$, we have
%  Every KDE whose kernel $k$ satisfies $k'\in L_1(\RR)$ is subject to
 \begin{equation*}
   \var[\hat{f}_n(x)] \leq  2 h^{-2} k(0) K_{n}.
 \end{equation*}
\end{lemma}

\begin{proof}
We will prove the inequality
\begin{equation}
\label{eqn:varIneq}
 \var[\hat{f}_n(x)] \leq K_{n} {h^{-4}} \left( \int_{\RR} k'\left( \frac{x-z}{h} \right) \d z  \right)^{2}.
\end{equation}
The result follows from this inequality by making the change of variable $w = (x-z)/h$ in the integral.
To prove \cref{eqn:varIneq}, we rewrite the variance of $\hat{f}_n(x)$ using integration by parts as follows: 
\begin{align}
%   \var[\hat{f}_n(x)]
%   =&
  \EE\left[ \left| \hat{f}_n(x) - \EE [\hat{f}_n(x)]\right|^{2}\right] 
  =&
  {h^{-2}}\, \EE\left[ \left|\int_{\RR}k\left( \frac{x-z}{h} \right)\d \hat{F}_n(z) - \int_{\RR}k\left( \frac{x-z}{h} \right)\d{F}(z) \right|^{2}\right] \nonumber\\
  =&
  h^{-4}\,\EE \left[ \left|\int_{\RR}\left( \hat{F}_n(z) - F(z) \right) k'\left( \frac{x-z}{h} \right)\d z \right|^{2}\right] \nonumber\\
  =&
 {h^{-4}} \,\EE \left[\int_{\RR}\int_{\RR} \Delta_n(y) \Delta_n(z) k'\left( \frac{x-y}{h} \right) k'\left( \frac{x-z}{h} \right)  \d y\d z\right].
 \label{eqn:varIneqCont}
\end{align}
Observe that $\EE[\Delta_n(z)] = 0$ since $\EE[\hat{F}_n(z)] = F(z)$.  
%  $$\EE[\hat{F}_X(z)] = \EE\left[\frac{1}{n}\sum_{i=0}^{n-1}\mathds{1}_{\{X_i\leq z\}}\right] = \PP[X\leq z] = F_X(z)$$
Consequently,
$$ \EE[\Delta_n(y)\Delta_n(z)] = \Cov[\Delta_n(y),\, \Delta_n(z)] \leq K_{n}. $$
Finally, \cref{eqn:varIneq} follows by interchanging the expectation and the integrals in 
\cref{eqn:varIneqCont}.
\end{proof}

For all $x\in \RR$, define $\bH(x) = \{\bu \in [0,1)^s : g(\bu) \le x\}$ and 
its complement $\overline{\bH}(x) = [0,1)^s \setminus \bH(x)$.
Under \cref{ass:strat}, let $\cB(x)$ be the set of subcubes $S_{\bi}$ that have a nonempty intersection
with both $\bH(x)$ and $\overline{\bH}(x)$.
The next lemma bounds the cardinality of $\cB(x)$ when $g$ is nondecreasing.

\hpierre{For any point set $P_n$ and $x\in\RR$, we have $F(x) = \PP(g(\bu) \le x) = \vol(\bH(x))$ and 
$\hat F_n(x) = |P_n\cap \bH(x)|/n$, so $\Delta(x) = |P_n\cap \bH(x)|/n - \vol(\bH(x))$.}

%%%%%%%%%%%%%%%%%%%%

\begin{lemma}
\label{lem:stratStrings}
Under \cref{ass:strat,ass:gmonotone}, for all $x\in\RR$, $|\cB(x)| \le s n^{(s-1)/s}$.
\end{lemma}

\begin{proof}
 Let $\bI_0 = \{\bi = (i_1,i_2,\dots,i_s) \in\bI$ such that $\min_{j} i_j = 0\}$, which is
 the set of indices $\bi\in\bI$ for which at least one face of $S_{\bi}$ lies on a face of 
 $[0,1)^{s}$ that contains the origin. 
 Each of the $s$ faces of $[0,1)^{s}$ touches at most $b^{s-1}$ elements of $\bI_0$
 and therefore $|\bI_0| \le s q^{s-1} = s n^{(s-1)/s}$.
 Now, for any $\bi = (i_1,i_2,\dots,i_s) \in\bI_0$, consider the \emph{diagonal string} of subcubes 
 $S_{\bi'(k)}$ with $\bi'(k) = (i_1 +k,i_2+k,\dots,i_s+k)$ for $0\le k < q-\max_j i_j$.
 We argue that for any $x\in\RR$, at most one subcube in this diagonal string can belong to $\cB(x)$.
 Indeed, suppose that two distinct subcubes in the string belong to $\cB(x)$, say 
 $S_{\bi'(k_1)}$ and $S_{\bi'(k_2)}$ for $k_1 < k_2$.  
 Since both subcubes contain points from $\bH(x)$ and $\overline{\bH}(x)$, there must be two points
 $\bu_1 \in S_{\bi'(k_1)} \cap \overline{\bH}(x)$ and $\bu_2\in S_{\bi'(k_2)}\cap\bH(x)$.
 This implies that $g(\bu_2) \le x < g(\bu_1)$ while $\bu_1 < \bu_2$ coordinatewise,
 which contradicts the assumption that $g$ is nondecreasing.
 Since there are no more than $s n^{(s-1)/s}$ diagonal strings and each contains at most one element of $\cB(x)$,
 the result follows.
\end{proof}

\begin{lemma}
\label{lem:stratKn}
Under \cref{ass:strat,ass:gmonotone}, $K_n \le ({s}/{4}) h^{-2} n^{-(s+1)/s}$.
\end{lemma}

\begin{proof}
For each $\bi\in\bI$, consider the random variables
 \begin{equation*}
 \delta_{\bi}(x) = \left| P_n \cap \bH(x) \cap S_{\bi} \right| - n\,\vol(\bH(x) \cap S_{\bi}).
\end{equation*}
We make three observations. 
Firstly, $S_{\bi}$ contains exactly
one point of $P_n$ by \cref{ass:strat}. Consequently, each $\delta_{\bi}(x)$ is a Bernoulli 
random variable (with parameter $p=n\,\vol(\bH(x)\cap S_{\bi})$) minus its mean $p$ and, therefore, 
$\var[\delta_{\bi}(x)] = p(1-p) \leq 1/4$.  
Secondly, for each $\bi$  for which $S_{\bi} \not\in \cB(x)$, 
$\delta_{\bi}(x) = 0$, so $\var[\delta_{\bi}(x)] = 0$. 
Thirdly, for any two distinct subcubes, the positions of the points of $P_n$ in these subcubes are independent.
% we have assumed that the points of $P_n$ are uncorrelated.
As a consequence of these three observations we see that
\begin{align*}
 \var[\Delta_n(x) ]
 =& \var\left[ \frac{1}{n} \sum_{\bi \in\bI} {\delta_{\bi}(x)} \right] 
% = \frac{1}{n^2} \sum_{\bi \in \bI} \var[\delta_{\bi}(x)] 
 = \frac{1}{n^2} \sum_{\bi: S_{\bi} \in\cB(x)} \var[\delta_{\bi}(x)]
%  \\ 
 \leq
%  & 
 \frac{1}{4n^{2}} s n^{-(s+1)/s}.
\end{align*}
%  This bound is independent of $x$. 
By applying \Cref{lem:stratK,lem:stratStrings} we then obtain
%\begin{align*}
\[ % \displaystyle
 \var[\hat{f}_n(x)] ~\leq~ 
   2 h^{-2} k(0) K_n  ~\leq~ \frac{k(0)}{2 (hn)^{2}} |\cB(x)| 
   ~\leq~  \frac{s k(0)}{2}  h^{-2}  n^{-(s+1)/s}.
\]
%\end{align*}
\end{proof}

\begin{proof}[Proof of \cref{prop:stratification}]
Combining \cref{lem:stratK,lem:stratKn} and integrating the variance bound with respect to $x$ 
over $[a,b]$ yields the result.
%  with the additional factor $(b-a)$ in the IV bound.
\end{proof}

In the above arguments, we assumed that the strata were cubic, but this is not necessary.
We could instead partition $[0,1)^s$ into $n=\prod_{j=1}^s q_j$
cells congruent to $\prod_{j=1}^d[0,1/q_j)$, 
subject to the condition $\max_j q_j \le \lambda\min_j q_j$ for some $\lambda<\infty$. 
Then one can bound the cardinality of $\cB(x)$ and $\var[\hat F_n(x)]$ in a similar way.
Non-cubic strata make sense if $g$ varies more in some directions than in others.
Finally, our bounds are proved under the assumption that $g$ is monotone, 
but this assumption is \emph{not necessary} for stratification to improve the MISE 
and/or its convergence rate.
\section{Empirical Study}
\label{sec:empirical}

Our analysis in the previous sections was in terms of (asymptotic) bounds.
Here, we study the IV and MISE behavior from a different viewpoint:
our goal is to estimate empirically how they really behave in a range of values 
of $n$ and $h$ that one is likely to use. 
For this, we use a simple regression model to approximate the true IV and MISE
in the region of interest. 
For some examples, we estimate the model parameters from simulated data,
test the goodness of fit of the regression models in-sample and out-of-sample,   
%  GoF is good.  
and show how the model permits one to estimate the optimal $h$ as a function of $n$,
as well as the resulting MISE and its convergence rate, under RQMC.
\hpierre{Without this type of modeling, we would not know how to choose $h$ and it would be hard
to estimate the optimal MISE and its convergence.}

%  \bigskip\centerline{\LARGE\color{red}  **********  RENDU ICI  *********** }

%%%%%%%%%%%%%%%%%%%%%%%%%%%%%%%%%%%%%%%%%%%%%%%%%%%%%%%%%%%%%%%
\subsection{Experimental setting and regression models for the local behavior of the IV, ISB, and MISE}
\label{sec:experiment}
\label{sec:numerical}

We will use the following models to approximate 
the true IV and ISB in a limited range of values of $n$ and $h$ of interest:
%\art{I sent over a more formal version of this stuff.
% From memory I think our analysis needs to have $\alpha\ge1$
% which the relevant $\alpha$ are, but we should remark on it.}
%\pierre{Yes, but it was more complicated and I do not think 
% we need more formality here. 
% Also there is no asymptotic involved, so no $\epsilon$ is needed.}
\begin{equation}                                \label{eq:model-iv-isb}
  \IV \approx C n^{-\beta} h^{-\delta} 
     \quad\mbox{ and }\quad
  \ISB \approx B h^\alpha, 
\end{equation}
for positive constants $C$, $\beta$, $\delta$, and $B$, 
that can be estimated as explained below, and with $\alpha=4$.
\hflorian{Since we removed the histogram estimator, do we keep it like that or insert $\alpha=4$ right away. It might make things more difficult for the online appendix, though.}
% 
%  {and we take $\alpha = 2$ for the histogram and $\alpha=4$ for the KDE.}
%   by linear regression (in log scale).
This gives  
% in \cref{eq:model-iv-isb} gives
%\begin{equation}                                \label{eq:mise-approx-nh}
$ \MISE \approx C n^{-\beta} h^{-\delta} + B h^\alpha$.
%\end{equation}
The bounds derived in the previous sections have this form, and this 
motivates our model, but here we want to estimate the true values, 
which generally differ from the bounds.
Once the parameters are estimated, we can estimate the optimal $h$ by 
minimizing the MISE estimate for any given $n$ in the selected range.
% -- How to select h when we know them?
% for our setting where $\alpha\ge1$ and $\delta>0$
In our setting, this estimated MISE is a convex function of $h$.
Taking the derivative with respect to $h$ and setting it to zero yields
%\begin{equation}                                \label{eq:approx-h}
  $h^{\alpha+\delta} = [{C\delta}/{(B\alpha)}] n^{-\beta}$.
%\end{equation}
Thus, if we take $h = \kappa n^{-\gamma}$, the constants $\kappa$ and $\gamma$
that minimize the MISE (based on our model) 
are $\kappa = \kappa_* := (C\delta/B\alpha)^{1/(\alpha+\delta)}$ 
and $\gamma = \gamma_* := \beta/(\alpha + \delta)$.
Plugging them into the MISE expression gives
\begin{equation}                                \label{eq:mise-approx-n}
 \MISE \approx K n^{-\nu}
\end{equation}
with $K = K_* := C\kappa_*^{-\delta} + B\kappa_{*}^\alpha$ 
and $\nu = \nu_* := \alpha\beta/(\alpha+\delta)$.
If $h$ is taken too small (e.g., by taking $\kappa < \kappa_*$  or $\gamma > \gamma_*$
in the formula for $h$), 
% If we take $h = \kappa n^{-\gamma}$ with $\gamma= \gamma_*$ but $\kappa < \kappa_*$,
% then $h$ will be too small, 
the IV will be too large and will dominate the MISE, 
so we will observe a MISE that decreases just like the IV.
The opposite happens if $h$ is too large: the ISB dominates the MISE.

% -- How to estimate them?
To estimate the model parameters for IV, we take the log 
%  in the IV expression in \cref{eq:model-iv-isb}, we 
to obtain the linear model
\begin{equation}                                \label{eq:log-linear}
  \log(\IV) \approx \log C - \beta \log n - \delta \log h,
\end{equation}
and we estimate the parameters $C$, $\beta$, and $\delta$ by linear regression.
Since $n$ is always a power of 2 for our RQMC points, we take all the logarithms in base 2.
In our experiments, we selected a set of 36 pairs $(n,h)$ 
with $n = 2^{14}, \dots, 2^{19}$ and $h = h_0,\dots,h_5$
%  \hpierre{Always $\tau=5$?}
where $h_j = h_0 2^{j/2} = 2^{-\ell_0 + j/2}$ and $2\ell_0$ is an integer selected from pilot runs.
% (see the supplement for the details).
This selection of $\ell_0$ is the only step that requires human intervention.

For each $n$ and each point set (MC, Stratification or RQMC), we generate a sample of size $n$, 
sort the sample, and then compute the density estimator for each $h$, 
for this sample. 
\hflorian{Should we include stratification in ``(MC or RQMC)''. 
Because until now we have always pretty much separated stratification from RQMC.}
\hflorian{Actually, there is only one density estimator left\ldots}
That is, we use the same sample for all estimation methods and all $h$.
We make $n_r=100$ independent replications of this procedure, which gives us 
independent replicates of the density estimator for the selected pairs $(n,h)$.
To obtain an unbiased estimator of the integral that defines the IV,
we take a stratified sample of $n_e = 1024$ evaluation points over the interval $[a,b]$,
%   using the midpoint rule or using independent random points, 
compute the empirical variance of the KDE at each point, 
based on the $n_r$ replications, and take the average multiplied by $(b-a)$.
Larger values of $n_e$ gave about the same estimates.
% We also tried larger values of $n_e$ such as 4096, 20,357 (a prime number)
% and the estimates were about the same.

We approximate the ISB in \cref{eq:model-iv-isb} by the AISB, for which $\alpha=4$
and $B = (\mu_2(k))^2 R(f'')/4$.   
For the Gaussian kernel, used in all our experiments, 
%  \hflorian{In all our experiments we take $k$ as the Gaussian kernel.}
\Cref{ass:kGeneral,ass:kernel} are satisfied, and $\mu_2(k) = 1$.
% Estimating the constant $B$ is difficult because 
% we have no unbiased estimator for the bias when the exact density is unknown.  
We estimate the integral $R(f'')$ as explained in \Cref{sec:density-estimators},
using RQMC instead of MC to improve the accuracy.
% so we take $B = R(f'')/4$ for the KDE.

Once we have the estimates $\hat{\kappa}_*$ and $\hat{\gamma}_*$ of 
%  all parameters in the model (\cref{eq:mise-approx-n}),
$\kappa_*$ and $\gamma_*$, we test the models out-of-sample by making an
independent set of simulation experiments with pairs $(n,h)$ that satisfy 
$h = \hat{h}_*(n) := \hat{\kappa}_* n^{-\hat{\gamma}_*}$ (the estimated optimal $h$)
for a series of values of $n$.
\hflorian{As a matter of fact, we have never introduced the optimal $h$ as $h_*$!}
At each of these pairs $(n,h)$, we sample $n_r$ fresh independent replicates of the RQMC 
density estimator and compute the IV estimate, as well as the MISE estimate in the 
simple examples where the density is known.  
In the latter case, we fit again the linear regression model for $\log(\MISE)$
vs $\log n$ to re-estimate the parameters $K$ and $\nu$ 
in \cref{eq:mise-approx-n} and assess the goodness-of-fit.
In our results, we denote these new estimates by $\tilde K$ and $\tilde \nu$.
Of course, these model testing steps are not needed if one wishes to only estimate 
the density $f$ and not to study the convergence properties.

In the end, we also compare the efficiencies of different methods for the same example
by comparing their estimated MISE for $n = 2^{19}$ with the $h$ recommended by the model.
We denote by e19 the value of $-\log_2(\MISE)$ for $n = 2^{19}$; 
that is, we have $\MISE = 2^{-e19}$.
The efficiency gain of RQMC vs MC can be assessed by comparing their e19 values.

The point sets considered in our experiments were:
(1) independent points (MC);
(2) stratification of the unit cube (Stratif);
(3) a Sobol' point set with a left random matrix scrambling and random digital shift (Sobol'+LMS); and
% (6) Sobol'+LMS with a baker's transformation (Sobol'+LMS+baker),\\
(4) a Sobol' point set with nested uniform scrambling (Sobol'+NUS).
\hflorian{Can we assume that it is known what ``NUS'' stands for or should we write it in full at least once?}
% 
% \end{verse}
The last two are well-known RQMC point sets  \cite{vLEC18a,vOWE97b,vOWE03a}
and we view stratification as a weak form of RQMC.
The short names in parentheses are used in the plots and tables.
These point sets and randomizations are implemented in SSJ \cite{iLEC16j},
which we used for our experiments.
\hpierre{More details and results can be found in the online supplement.}

%  \mpierre{**********************   HERE   *********************************}

%%%%%%%%%%%%%%%%%%%%%%%%%%%%%%%%%%%%%%%%%%%%%%%%%%%%%   EX1
\subsection{A normalized sum of standard normals}
\label{sec:example-N01}

\hpierre{-- One can find $R(f'')$ when $f$ is the standard normal density apparently in \cite{tWAN95a}.
   We also know that $\mu_0(k) = 1/2\sqrt{\pi} \approx 0.2821$ for the Gausian kernel; see \cite{tSCO15a}, page 152.
	 Check also \cite{tMAR92a}.}

As in \cite{tOWE17a}, we construct a set of test functions with arbitrary dimension $s$ 
and for which the density $f$ of $X$ is always the standard normal,
$f(x) = \exp(-x^2/2)/\sqrt{2\pi}$ for any $s$.
For this, let $Z_1,\dots,Z_s$ be $s$ independent standard normal random variables generated by inversion 
and put $X = (a_1 Z_1 + \cdots + a_s Z_s)/\sigma$, where $\sigma^2 = a_1^2 + \cdots + a_s^2$.
For this simple example, the density is already known, so there is no need to estimate it,
but this is convenient for testing the methodology, since it permits us to compute and compare unbiased 
estimators of the IV, ISB, and MISE for both MC and RQMC.
For MC, these quantities do not depend on $s$, but for RQMC, the IV and MISE do depend on $s$,
and we want to see in what way.

We can also compute $R(f'')$   % = \int_{-b}^b f''(x)\d x$ 
exactly in this example, 
which means we can compute $B$ for the AISB and the asymptotically optimal $h$ for the AMISE.
However, we will first make experiments as if we did not know this $B$ and have to estimate it,
%  After that, we will 
and then compare our estimates with the exact $B$.
%  we will also estimate this integral via KDE and assess the error.  
Here, $g$ is a monotone increasing function, so \Cref{cor:miseStrat} applies when we use stratification.
% use stratified points that satisfy \Cref{ass:strat}.
\Cref{ass:g} holds only if we truncate the normal distributions of the $Z_j$, %  which we can do,
but it makes no significant difference on our empirical results if the truncated range contains 
the interval $[-8,8]$, for example, so from the practical viewpoint, we can ignore it.
\hflorian{The last sentence seems a little misplaced. Maybe we should put it towards the end of the experiment, when we actually compare the theoretical and empirical results.}

\begin{table}[!htbp] %KDE [-2,2]
 \centering
\caption{Parameter estimates for the KDE, for a sum of normals, over $[-2,2]$.}
% \caption{Parameter estimates of the regression model with a KDE ($\alpha = 4$) over $(-2,2)$}
\label{tab:sumnormalnormalized-parameters-kde}
\small
\begin{tabular}{l | l l l l l l l l l l l }
% \multirow{2}{*}{}	&& \multicolumn{1}{c}{$s = 1$}	& \multicolumn{2}{c}{$s = 2$}	
%      & \multicolumn{2}{c}{$s = 3$}	& \multicolumn{2}{c}{$s = 5$} 	& \multicolumn{1}{c}{$s = 10$}	&{$s = 20$}\\
			& MC		& NUS		& LMS		& NUS		& LMS		& NUS		& LMS		& NUS		& NUS				& NUS	 & NUS \\
\hline
  $s$       &         &  1      &  2      &  2      &  3      &  3    &  5    &  5      &  10   &  20  & 100 \\
 $\ell_0$		& 4.5		  & 8.5		  & 6.0		  & 6.0		  & 5.0		  & 5.0		& 4.5		& 4.5		  & 4.0			& 4.0 & 4.0 \\
 $C$		  	& 0.265		& 0.032		& 0.243		& 0.212		& 0.144		& 0.180		& 0.140		& 0.096		& 0.029	& 0.078 & 0.079\\
 $\beta$		& 1.038		& 2.791		& 2.112		& 2.101		& 1.786		& 1.798		& 1.301		& 1.270		& 1.011	& 0.996 & 1.010\\
 $\delta$		& 1.134		& 3.004		& 3.196		& 3.196		& 3.383		& 3.357		& 2.295		& 2.303		& 1.811	& 1.421 & 1.463\\
 $R^2$			& 0.999		& 0.999		& 1.000		& 1.000		& 0.995		& 0.995		& 0.979		& 0.978		& 0.990	& 0.991 & 0.996\\
%  $B$		    & 0.042		& 0.042		& 0.042		& 0.042		& 0.042		& 0.042		& 0.042		& 0.042		& 0.042	& 0.042 & 0.042\\
 \hline
 $\hat{\kappa}_*$	& 1.121		& 0.925		& 1.238		& 1.215		& 1.156		& 1.191		& 1.109		& 1.045		& 0.820	& 0.925 & 0.934\\
 $\hat{\gamma}_*$	& 0.202		& 0.398		& 0.293		& 0.292		& 0.242		& 0.244		& 0.207		& 0.201		& 0.174	& 0.184 & 0.185\\
 $\ell_*$		      & 3.675		& 7.682		& 5.268		& 5.266		& 4.386		& 4.391		& 3.776		& 3.765		& 3.590	& 3.604 & 3.612\\
% \hline
 $\hat{K}_*$		& 0.299		& 0.071		& 0.221		& 0.205		& 0.163		& 0.184		& 0.173		& 0.137		& 0.061		& 0.117 & 0.119\\
 $\hat{\nu}_*$		& 0.808		& 1.594		& 1.174		& 1.168		& 0.967		& 0.978		& 0.826		& 0.806		& 0.696	& 0.735 & 0.740\\
% \hline
 $\tilde{\nu}$		& 0.781		& 1.595		& 1.176		& 1.169		& 0.976		& 0.975		& 0.832		& 0.806		& 0.744	& 0.764 & 0.774\\
% $R^{2}_{\IV}$		& 0.997		& 0.995		& 1.000		& 1.000		& 0.977		& 0.978		& 0.909		& 0.917		& 0.983	& 0.993 & 0.983\\
% $R^{2}_{\ISB}$		& 0.969		& 0.995		& 0.991		& 0.991		& 0.988		& 0.989		& 0.985		& 0.985		& 0.972	& 0.983 & 0.986\\
% $R^{2}_{\MISE}$	& 0.994		& 0.997		& 0.997		& 0.998		& 0.980		& 0.981		& 0.937		& 0.942		& 0.981	& 0.991 & 0.983\\[3pt]
 \hline
 e19			       & 17.01		& 34.06		& 24.39		& 24.38		& 20.79		& 20.80		& 17.88		& 17.79		& 17.28	& 17.07 & 17.05
\end{tabular}
\end{table}

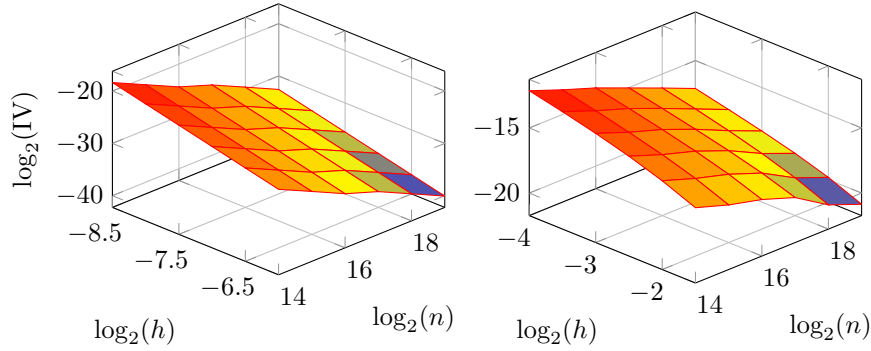
\begin{figure}[!htbp] %IV 3d-plots
\small
\centering
  \begin{tikzpicture}
   \begin{axis}[ 
      xlabel=$\log_2(h)$,
      ylabel=$\log_2(n)$,
      zlabel=$\log_2(\IV)$,
			xtick={-8.5, -7.5, -6.5},
      grid,
      width = 6cm,
      view={45}{35}
      ] 
          \addplot3 [ surf, mesh/rows=6, faceted color=red]coordinates { 
  (     -8.5 , 14.0  ,-18.37537827520431) 
  (     -8.0 , 14.0  ,-19.876120327280525) 
  (     -7.5 , 14.0  ,-21.37205632611671) 
  (     -7.0 , 14.0  ,-22.876421481330194) 
  (     -6.5 , 14.0  ,-24.391678301312837) 
  (     -6.0 , 14.0  ,-25.89607923261192) 
  (     -8.5 , 15.0  ,-21.366222596893614) 
  (     -8.0 , 15.0  ,-22.874401137177774) 
  (     -7.5 , 15.0  ,-24.381707801126375) 
  (     -7.0 , 15.0  ,-25.87514442697066) 
  (     -6.5 , 15.0  ,-27.37999397563076) 
  (     -6.0 , 15.0  ,-28.878081024759407) 
  (     -8.5 , 16.0  ,-24.355149311211374) 
  (     -8.0 , 16.0  ,-25.857622330724677) 
  (     -7.5 , 16.0  ,-27.35551984803196) 
  (     -7.0 , 16.0  ,-28.864551729561622) 
  (     -6.5 , 16.0  ,-30.37197715253164) 
  (     -6.0 , 16.0  ,-31.888358220914046) 
  (     -8.5 , 17.0  ,-26.564234469135272) 
  (     -8.0 , 17.0  ,-28.07652871814744) 
  (     -7.5 , 17.0  ,-29.579465728485545) 
  (     -7.0 , 17.0  ,-31.0822369574906) 
  (     -6.5 , 17.0  ,-32.57536432838298) 
  (     -6.0 , 17.0  ,-34.064156805163115) 
  (     -8.5 , 18.0  ,-29.55820195353696) 
  (     -8.0 , 18.0  ,-31.061441025331245) 
  (     -7.5 , 18.0  ,-32.56830458841306) 
  (     -7.0 , 18.0  ,-34.06500361469599) 
  (     -6.5 , 18.0  ,-35.558613272124674) 
  (     -6.0 , 18.0  ,-37.05769086836235) 
  (     -8.5 , 19.0  ,-32.57008558530116) 
  (     -8.0 , 19.0  ,-34.05921483146138) 
  (     -7.5 , 19.0  ,-35.563896076087204) 
  (     -7.0 , 19.0  ,-37.060833294157895) 
  (     -6.5 , 19.0  ,-38.56228016923607) 
  (     -6.0 , 19.0  ,-40.06412978606427) 
      }; 
   \end{axis}
  \end{tikzpicture}
\begin{tikzpicture}
    \begin{axis}[ 
      xlabel=$\log_2(h)$,
      ylabel=$\log_2(n)$,
%       zlabel=$\log_2(\IV)$,
      grid,
      width = 6cm,
      view={45}{35}
      ] 
  \addplot3 [ surf, mesh/rows=6, faceted color=red]coordinates { 
  (     -4.0 , 14.0  ,-12.141834632610525) 
  (     -3.5 , 14.0  ,-12.714513626079922) 
  (     -3.0 , 14.0  ,-13.346094932438039) 
  (     -2.4999999999999996 , 14.0  ,-14.06307309858084) 
  (     -1.9999999999999998 , 14.0  ,-14.911584108583375) 
  (     -1.4999999999999996 , 14.0  ,-15.947400731837604) 
  (     -4.0 , 15.0  ,-13.103862095868129) 
  (     -3.5 , 15.0  ,-13.684826759251795) 
  (     -3.0 , 15.0  ,-14.300644295309743) 
  (     -2.4999999999999996 , 15.0  ,-14.977611699693156) 
  (     -1.9999999999999998 , 15.0  ,-15.773306632828346) 
  (     -1.4999999999999996 , 15.0  ,-16.770325286366184) 
  (     -4.0 , 16.0  ,-13.99458852355793) 
  (     -3.5 , 16.0  ,-14.543258398555217) 
  (     -3.0 , 16.0  ,-15.118858756804652) 
  (     -2.4999999999999996 , 16.0  ,-15.749818798351592) 
  (     -1.9999999999999998 , 16.0  ,-16.48586852638785) 
  (     -1.4999999999999996 , 16.0  ,-17.388883378247687) 
  (     -4.0 , 17.0  ,-15.031478253996914) 
  (     -3.5 , 17.0  ,-15.57100564977963) 
  (     -3.0 , 17.0  ,-16.138068556828937) 
  (     -2.4999999999999996 , 17.0  ,-16.746089478709596) 
  (     -1.9999999999999998 , 17.0  ,-17.429748041463892) 
  (     -1.4999999999999996 , 17.0  ,-18.273496323708518) 
  (     -4.0 , 18.0  ,-16.073697267246246) 
  (     -3.5 , 18.0  ,-16.650979068908587) 
  (     -3.0 , 18.0  ,-17.283013678324764) 
  (     -2.4999999999999996 , 18.0  ,-17.997915019017473) 
  (     -1.9999999999999998 , 18.0  ,-18.835348609068692) 
  (     -1.4999999999999996 , 18.0  ,-19.870143090441466) 
  (     -4.0 , 19.0  ,-17.145782470699555) 
  (     -3.5 , 19.0  ,-17.729523339107285) 
  (     -3.0 , 19.0  ,-18.342643968074906) 
  (     -2.4999999999999996 , 19.0  ,-19.022371148685984) 
  (     -1.9999999999999998 , 19.0  ,-19.82797923461081) 
  (     -1.4999999999999996 , 19.0  ,-20.857841500986858) 
      }; 
   \end{axis}
\end{tikzpicture}
  \caption{$\log_2(\IV)$ for the KDE with Sobol'+NUS for $s=1$ (left) and $s=20$ (right).}
  \label{fig:normal-iv-plots-3d}
\end{figure}

\pgfplotscreateplotcyclelist{defaultcolorlist}{%
	{blue!95!black,line width=0.9pt,mark=dot*,solid},
	{red!95!black,line width=0.9pt,mark=square*,solid},
	{green!98!black,line width=0.9pt,mark=triangle*,solid},
	{black,line width=1.0pt,mark=star,loosely dotted},
	{brown!85!black,mark=square,dashed},
	{purple!85!black,mark=triangle,dotted}}

%%  Plots of beta, delta, and e19 with n_e =1024
%%
\begin{figure}[!htbp] 
\small
\centering
 \begin{tikzpicture} \footnotesize
    \begin{axis}[ 
		cycle list name=defaultcolorlist,
    legend style={at={(1.02,1.45)}, anchor={north east}},
      xtick={1,2,3,4,5},
      xlabel=$s$,
       ylabel=$\beta$,
       grid,
      width = 5cm,
      ] 
      \addplot table[x=d, y=e19] { 
      d  	e19
      1.0   1.0375360933519602 
      2.0  	0.9817346869171588
      3.0  	1.0024735632771342
      4.0	1.0079356772531212
      5.0	1.0098464641661604
%       10	0.9769991999220669
%       20	1.0030462351095661
      }; 
      \addlegendentry{MC}
    \addplot table[x=d, y=e19] { 
      d  	e19
      1.0   3.002964157594687 
      2.0	1.674023021547137
      3.0	1.380082095661359
      4.0	1.2404245016455961
      5.0	1.1490586640557363
      }; 
      \addlegendentry{Strat}

      \addplot table[x=d, y=e19, color=green] { 
      d  	e19
      1.0   3.093222756847555
      2.0	2.111926217575321
      3.0	1.7856402926608306
      4.0	1.5396170904691246
      5.0	1.3006898210744569
%       10	1.047662053767881
%       20	1.0013166351278098
      }; 
      \addlegendentry{S+LMS}
      \addplot table[x=d, y=e19] { 
      d  	e19
      1.0   2.790730716599069
      2.0	2.1006370956240703
      3.0	1.7981987663654881
      4.0	1.528617248734558
      5.0	1.2700687074532768
%       10	1.0105962306711813
%       20	0.9962779641610845
      }; 
      \addlegendentry{S+NUS}
    \end{axis}
  \end{tikzpicture}
 \begin{tikzpicture} \footnotesize
    \begin{axis}[ 
		cycle list name=defaultcolorlist,
    legend style={at={(1.3,1.5)}, anchor={north east}},
      xtick={1,2,3,4,5},
      xlabel=$s$,
       ylabel=$\delta$,
       grid,
      width = 5cm,
      ] 
      \addplot table[x=d, y=e19] { 
      d  	e19
      1.0   1.13374277298603 
      2.0  	1.145446900262749 
      3.0  	1.1533989941404468
      4.0	1.147698979356037
      5.0	1.146271176919399
%       10	1.1101054572729412
%       20	1.1293339806633391
      }; 
      \addlegendentry{MC}
    \addplot table[x=d, y=e19] { 
      d  	e19
      1.0   3.0019666686532966 
      2.0	2.3313257330242676
      3.0	2.1474460624167406
      4.0	2.018178303008548
      5.0	1.7772856950294127 
      }; 
      \addlegendentry{Strat}

      \addplot table[x=d, y=e19, color=green] { 
      d  	e19
      1.0   3.230161030420311
      2.0	3.195922010622765
      3.0	3.383409935126675
      4.0	3.0475951124178597
      5.0	2.295483684853652
%       10	1.8626895332468445
%       20	1.4712589063223538
      }; 
      \addlegendentry{S+LMS}
      \addplot table[x=d, y=e19] { 
      d  	e19
      1.0   3.0038390129161554
      2.0	3.1960100358542847
      3.0	3.357010159838343
      4.0	3.034734563708854
      5.0	2.3033112971451324
%       10	1.8106534558029388
%       20	1.421152145972232
      }; 
      \addlegendentry{S+NUS}
    \end{axis}
  \end{tikzpicture}
 \begin{tikzpicture} \footnotesize
    \begin{axis}[ 
		cycle list name=defaultcolorlist,
    legend style={at={(1.02,1.45)}, anchor={north east}},
      xtick={1,2,3,4,5},
      xlabel=$s$,
			ylabel=e19,
      grid,
      width = 5cm,
      ] 
     \addplot table[x=d, y=e19] { 
      d  	e19
      1.0   17.00659
      2.0  	16.98827 
      3.0  	17.00639
      4.0	16.93194
      5.0	16.89269
%       10	16.87066
%       20	16.99538
      }; 
      \addlegendentry{MC}
     \addplot table[x=d, y=e19] { 
      d  	e19
      1.0   34.52727 
      2.0	22.51723
      3.0	19.61043
      4.0	18.35518
      5.0	17.71156
      }; 
      \addlegendentry{Strat}
      \addplot table[x=d, y=e19, color=green] { 
      d  	e19
      1.0   34.10130
      2.0	24.38811
      3.0	20.79490
      4.0	18.81637
      5.0	17.87857
%       10	17.19167
%       20	17.10460
      }; 
      \addlegendentry{S+LMS}
      \addplot table[x=d, y=e19] { 
      d  	e19
      1.0   34.05621
      2.0	24.37646
      3.0	20.79903
      4.0	18.86245
      5.0	17.79219
%       10	17.27756
%       20	17.06643
      }; 
      \addlegendentry{S+NUS}
    \end{axis}
  \end{tikzpicture}
  \caption{Estimated $\beta$, $\delta$, and e19 %  ($= -\log_2(\MISE)$ for $n = 2^{19}$) 
% 	   for the histogram (above)  and the KDE (below) over $(-2,2)$, 
			with MC, Stratification, Sobol'+LMS, and Sobol'+NUS.}
  \label{fig:normal-plots}
\end{figure}
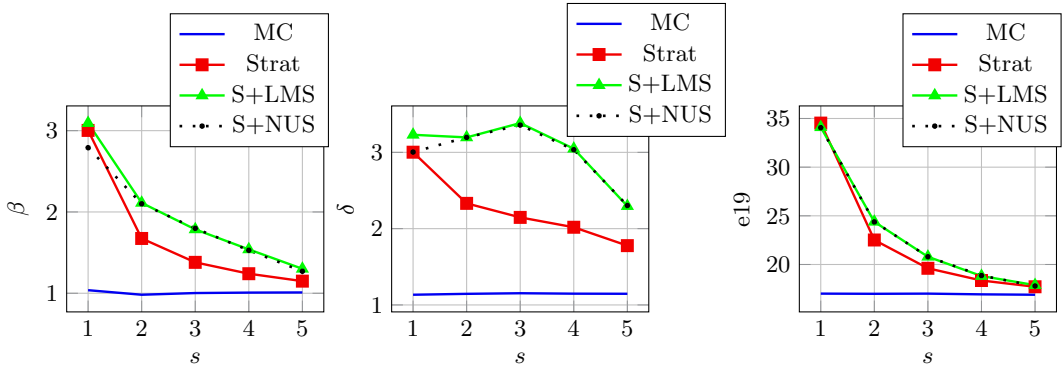

We estimate the density over $[a,b] = [-b,b] = [-2,2]$.
In our first experiment, we take $a_1 = \cdots = a_s = 1$, so all the coordinates have the same importance
(which is disadvantageous for RQMC).
Later, we will consider varying coefficients $a_j$. 
\Cref{tab:sumnormalnormalized-parameters-kde} 
summarizes the results when $B$ is estimated.
\hflorian{Coming back to an earlier discussion: we might save some space by removing $B$ from the table.}
% 
%  when we take $B$ as the exact asymptotic value.
For MC, our estimates given in the first column are based on experiments made with $s=1$, 
but are valid for all $s$, because the IV and ISB do not depend on $s$.
The estimated values for MC agree with the theory:
the exact asymptotic values are 
%  $\gamma = 0.4$ and $\nu = 2/3$ for the histogram, 
$\gamma = 0.2$, $\nu = 0.8$, and $\beta = \delta = 1$.
\hflorian{Ok, this is almost ridiculous to mention, 
but for the sake of uniformity: should we put it like $\gamma = 1/5$, $\nu = 4/5$ or $\gamma = 0.2$, $\nu = 0.8$?}%
The other columns give some results for Sobol'+LMS and Sobol'+NUS, for selected values of $s$.
For all $s > 1$ that we have tried, LMS and NUS give almost the same values.
%  and we skip some LMS results for this reason.
The first rows give the dimension $s$, the $\ell_0$ found by pilot runs and used to fit
the IV model, the estimated parameters $C$, $\beta$, and $\delta$ of the IV model,
the fraction $R^2$ of variance explained by this model, and the estimated $B$. 
The other quantities are defined in \Cref{sec:experiment},
except for $\ell_* = -\log_2 \hat h_*(2^{19})$, which gives an idea of the optimal $h$ for $n=2^{19}$.

Recall that the rates $\tilde\nu$ and e19 are obtained from a second-stage experiment,
by using the estimated $\hat h_*(n)$ from the model in the first stage.
All the $R^2$ coefficients are close to 1, which means that 
the log-log linear model is reasonably good in the area considered.
%  e19 is an abbreviation of $-\log_2(\MISE)$ for $n = 2^{19}$.
The estimate of $B$ is $B\approx 0.0418$ (same first three digits) for all $s$ and all RQMC methods.
%  (up to three decimal digits).
%  $B\approx 0.01081$ for the histogram and $B\approx 0.0418$ for the KDE.
Thus, this estimator of $B$ has very little variance.
The MISE reduction of RQMC vs MC can be assessed by comparing their values of 
e19 given in the last row.
For example, with the KDE for $s=1$, the MISE for $n=2^{19}$ is approximately $2^{-34}$ 
for Sobol'+NUS compared to $2^{-17}$ for MC, i.e., about $2^{17} \approx 125{,}000$ times smaller.
For $s=2$, for both LMS and NUS, the MISE is about $2^{-24.3}$, which is about 150 times
smaller than for MC. 

\Cref{fig:normal-iv-plots-3d} gives a visual assessment of the fit of the linear 
model for $\log_2(\IV)$ in the selected region, for two values of $s$.
We made similar plots for several $s > 1$ and all point sets, and the linear approximation looked reasonable
in all cases.  
\hpierre{For Sobol' points with a random digital shift only, in $s=1$ dimension, something different happens and the 
MISE becomes extremely small; this is discussed in the online supplement.}
\hflorian{As far as I remember, nothing peculiar happened for Sobol'+LMS. We only observed this small variance when  the points were equally spaced. This is the case when we restrict LMS to the first 19 bits, because then we only permute the points or when we take Sobol'+digital shift, i.e. lattice+shift.}%
% 
%  the MISE becomes extremely small.
%
\Cref{fig:normal-plots} shows the estimated $\beta$, $\delta$, and e19, for $s=1,\dots,5$,
for various point sets.  Stratification, shown here and not in the table, 
is exactly equivalent to Sobol'+NUS for $s=1$, and somewhat less effective for $s > 1$.

One important observation from the plots and the last row of the table (e19)
is that for all $s$, the RQMC methods never have a larger MISE than MC.
Their MISE is much smaller for small $s$, and becomes almost the same as for MC
when $s$ gets large.  The MISE rate $\tilde\nu$ behaves similarly.
% RQMC brings more MISE reduction for the KDE than for the histogram, as we would
% expect based on the smoothness of the KDE integrand and lack thereof for the histogram.
Another important observation is that the coefficients $\beta$ and $\delta$
in the IV model (which are both 1 with MC) are \emph{both} larger than 1 with RQMC.
For small $s$, with RQMC, $\beta$ is significantly larger than $\tilde\nu$, 
which means that the IV converges much faster as a function of $n$ when $h$ is fixed than 
when $h$ varies with $n$  to optimize the MISE.
This is explained by the large values of $\delta$, sometimes even larger than 3, 
which indicate that reducing $h$ to reduce the ISB increases the IV rapidly,
and this limits the MISE reduction that we can achieve.
\iffalse  %%%%%%%%%%
\hpierre{Recall that our asymptotic upper bound for the one-dimensional case with NUS in 
\Cref{sec:rqmc-iv} gave $\beta = \delta = 3$ and $\nu_* = 12/7 \approx 1.714$.
The values in the table are not far from these numbers.}%
% 
\hflorian{For now, we don't have the $s=1$ result anymore. 
But for $s=2,3,4,5$! For $s=1$, the stratification bound yields $\nu_*=4/3$, just as the KH bound.}%
% 
% As we have already mentioned above, the conditions of \cref{cor:miseStrat} are
% fulfilled for Stratification and so we can compare the empirical results with
\cref{tab:stratComparison} compares the empirical AMISE rates $\hat\nu_*$ with the rates $\nu$ 
in the bound of \cref{cor:miseStrat}, for stratification, for $s\leq5$.
The empirical values match the bounds very closely, 
which suggests that the bound in \cref{cor:miseStrat} is very tight.

\begin{table}[!hbtp]
 \centering
 \small
 \caption{Comparison of theoretical and empirical MISE-rates $\nu$ and $\hat\nu_*$ with stratification, 
   for a sum of normals.}
 \label{tab:stratComparison}
 \begin{tabular}{c|ccccc}
 $s$ & 1 & 2 & 3 & 4 & 5\\
 \hline
 $\nu$    		& -1.714 & -1.000 & -0.889 & -0.833 & -0.800 \\
 $\hat\nu_*$	& -1.715 & -1.064 & -0.899 & -0.831 & -0.802
\end{tabular}
\end{table}
\fi  %%%%%%%%%%%

Here $f$ is the standard normal density and
%  $R(f') = ({-b e^{-b^2}+\int_{0}^{b}e^{-x^2}\d x})/{2\pi}$
$R(f'') = [{ - b\left(2b^2- 1 \right) e^{-b^2} 
        + 3 \int_{0}^{b} e^{-x^2}\d x}]/{4\pi}$. 
For $b=2$, this gives $R(f'')\approx 0.19018$, so the true constant $B$ in the AISB is 
$B = R(f'')/4 \approx 0.04754$, whereas our estimate was 0.0418 
for all $s$ and all point sets.
The difference is not due to noise, but is a bias coming from the fact that we estimated 
$R(f'')$ via KDE with finite $n$.
We verified empirically that when we estimate these quantities with a larger $n$, 
the bias decreases slowly and appears to converge to 0 when $n\to\infty$.
%  In other words, the ISB in the region of interest differs from the AISB.
%
\hflorian{Apparently, the bias in the estimation of $B$ is larger than the variance 
  and converges to 0 very slowly for $n\to\infty$.}
\hpierre{We really need to understand and explain where this difference comes from.
  What $n$ did you use for the estimation?  
	Is the estimate sensitive to $n$ and $h$?  That is, what happens when you use a larger $n$ 
	or you change $h$ a little bit?  We need to understand what is the source of error
  in estimating $R(f')$ and $R(f'')$, especially because it is not noise.  
	Are you sure the exact formulas are correct?} 

We repeated the density estimation experiment by using the exact values of $B$ instead of the 
estimated ones to choose $h$, and the results were very close for all $s$. 
In particular, the MISE rates $\tilde\nu$ and the values of e19 were almost the same.

%%%%%%%%%%%%%%%

We now take different coefficients (weights) $a_j$ in the linear combination of the $Z_j$ that defines $X$.
Our purpose is to illustrate that there are situations where RQMC can perform very
well with the KDE even when the dimension $s$ is large.  This can occur for example if the effective
dimension is not large; i.e., when $g(\bu)$ depends mostly on just a few coordinates of $\bu$,
and does not vary much with respect to the other coordinates \cite{vCAF97a,vLEC09f}.
To illustrate this, we take $a_j = 2^{-j}$ for $j=1,\dots,s$, and we repeat the same set of experiments
as we did for equal weights, to estimate the density over $[-2,2]$.

\begin{table}[!bhtp] %KDE [-2,2]
 \centering
\caption{Parameter estimates for the KDE under Sobol'+LMS, for a weighted sum of normals with $a_j=2^{-j}$.}
% \caption{Parameter estimates of the regression model with a KDE ($\alpha = 4$) over $[-2,2]$}
\label{tab:sumnormalnormalized-weighted-parameters-kde}
\small
\begin{tabular}{l |l l l l l l l  }
  $s$      		& MC		& 2		& 4		& 10		& 20		& 50		& 100   \\
\hline
  $C$		  	& 0.171		& 0.173		& 0.038		& 6.7E-3	& 8.0E-3	& 7.3E-3	& 7.9E-3\\
 $\beta$		& 1.000		& 2.100		& 1.650		& 1.420		& 1.427		& 1.425		& 1.429\\
 $\delta$		& 1.137		& 3.189		& 3.745		& 3.626		& 3.582		& 3.604		& 3.603\\
 \hline
 $\hat{K}_*$		& 0.213		& 0.183		& 0.080		& 0.032		& 0.035		& 0.033		& 0.035\\
 $\hat{\nu}_*$		& 0.779		& 1.168		& 0.852		& 0.745		& 0.753		& 0.750		& 0.752\\
 $\tilde{\nu}$		& 0.774		& 1.176		& 0.892		& 0.750		& 0.730		& 0.758		& 0.752\\
 \hline
 e19			& 16.96		& 24.76		& 19.71		& 18.96		& 18.98		& 18.99		& 19.04
\end{tabular}
\end{table}

\cref{tab:sumnormalnormalized-weighted-parameters-kde} summarizes our findings for Sobol'+LMS,
for $s$ up to 100.  The results with Sobol'+NUS are very similar.
For $s=1$, the results are obviously the same as for our previous setting,
but they diverge when we increase $s$.
For example, in the previous setting, the MISE estimate with $n=2^{19}$ for $s=2$, 10, and 100, was
$2^{-24.38}$, $2^{-17.28}$, and $2^{-17.05}$, respectively, whereas with the new weights, it is
$2^{-24.76}$, $2^{-18.96}$, and $2^{-19.04}$, respectively.
For $s=100$, in particular, the MISE with RQMC and $n=2^{19}$ was about the same as for MC 
in the previous setting, and it is reduced by a factor of 4 in the present setting.
We also see from the table that in 10 or more dimensions, the convergence rate of the MISE is 
not improved, but the constant is improved (empirically).
As expected, when $s$ increases beyond about 10, all the model parameters appear to stabilize 
as a function of $s$.  In the previous setting, they were stabilizing around the MC values,
but now they stabilize to different values.  For example, in $s=100$ dimensions, $\beta$ was near
the MC value of 1, and now it is about 1.4.

\hflorian{
We further observe that with $a_j=2^{-j}$ we also gain over the unweighted case. This is not surprising
as the weights decrease very quickly so that the contributions of many summands to the error become extremely small. We can see
this in our experiment as the difference between the e19s of the weighted and the unweighted case increases from, roughly, 0.4 
in $s=2$ dimensions to 2 for $s=100$. 
The MISE gain of the weighted case compared to the unweighted case is caused by a reduction of the IV. Principally speaking,
these two models behave very similarly in that $\beta$ and $\delta$ change for $s\leq5$ and stabilize at certain values for $s>5$.
Contrary to the unweighted case, however, the IV parameters $\beta$ and $\delta$ do not approach the corresponding 
values for MC for large $s$. We see that even in 100 dimensions $\beta=1.4$, which is clearly beyond the MC rate.
Between $s=1$ and $s=5$ the values of $\beta$ decrease very slowly, while $\delta$ steadily increases.}

%%%%%%%%%%%%%%%%%%%%%%%%%%%%%%%%%%%%%%%%%%%%%%%%%%%%%%%%%%%			 			 
\subsection{Displacement of a cantilevel beam}
\label{sec:cantilever}

Bingham \cite{sBIN17a} gives the following simple model of the displacement $D$ of a cantilever 
beam with horizontal and vertical loads:
\begin{equation}       \label{eq:dformula}
	D = \frac{4L^3}{Ewt}\sqrt{\frac{Y^2}{t^4} + \frac{X^2}{w^4}}
\end{equation}
in which $L$ is the length of the beam, fixed to 100 inches, 
$w$ and $t$ are the width and thickness of the cross-section,
taken as 4 and 2 inches,
while $X$, $Y$, and $E$ are assumed independent and normally distributed
with means and standard deviations given as follows (in inches):
\smallskip
\begin{center}
		\begin{tabular}{lc|cc}
			\hline
			Description        & Symbol & Mean & St.\ dev.\ \\
			\hline
		%	Yield stress &  $R$ & $40{,}000$ & $2{,}000$ \\
			Young's modulus & $E$ & $2.9\times 10^7$ & $1.45\times 10^6$\\
			Horizontal load & $X$ & $500$ & $100$ \\
			Vertical load   & $Y$ & $1000$ & $100$ \\
			\hline\vspace{0.25em}
		\end{tabular}
\end{center}

We want to estimate the density of the relative displacement $\tilde X = D/D_0-1$, where $D_0=2.2535$ inches.
\hflorian{I just realized that our experiments used this normalized version (also proposed on Bingham's webpage)!}%
Here, the exact density is unknown, so unbiased estimators of the ISB and the MISE are not available,
but we can estimate the AISB as in the previous example, and use it to estimate the optimal $h$ and the MISE.
A plot of the estimated density, obtained with a  KDE with Sobol'+NUS and $n=2^{19}$ points, 
is given in \Cref{fig:canti-density}.
For the experiments reported here, we estimate the density of $\tilde X$ over the interval 
%  $(0.590,1.293)$, which excludes roughly 5 \% of the probability on each side, 
$[0.407,1.515]$, which covers about 99\% of the density (it excludes roughly 0.5 \% on each side). 
% We also tried the shorter interval $(0.590,1.293)$, which excludes 5\% of the density on each side,
% and the results were very similar, except that the estimated constant $B$ was about 
% 17\% smaller for the histogram and 25\% smaller for the KDE.

\begin{figure}[!htbp]  %density plot cantilever
\centering
 \begin{tikzpicture} \footnotesize
    \begin{axis}[ 
     x label style={at={(axis description cs:1.05,-0.03)},anchor=north},
%      xlabel near ticks,
			xlabel=$x$,
      ylabel=estimated density,
      xmin=0.40,
      xmax=1.50,
			ymin=0,
       grid,
       width = 8.4cm,
      height = 3.6cm,
      ] 
      \addplot table[x=x,y=density,color=red,mark= ,style=solid] { 
x 	density
0.41807	 0.08558167144426815
0.44024	 0.11332302977909534
0.4624	 0.14813434392133895
0.48456	 0.19065452655621257
0.50671	 0.2400963398666585
0.52888	 0.30186492821986105
0.55104	 0.37019128680379065
0.57319	 0.45100640839468614
0.59536	 0.542081284523286
0.61751	 0.6403779971227069
0.63968	 0.7494390420077861
0.66184	 0.8645022005120548
0.684	 0.985131750891621
0.70615	 1.1096320907982844
0.72832	 1.2329930255167805
0.75048	 1.3546412549115223
0.77264	 1.4685226219688443
0.7948	 1.5750350357062386
0.81695	 1.6672129445570856
0.83912	 1.7452761089872448
0.86128	 1.806515486525151
0.88344	 1.8492150362716004
0.9056	 1.8709919523314316
0.92775	 1.873383782390795
0.94991	 1.8541460917982937
0.97208	 1.8165247623699547
0.99424	 1.7613273961968763
1.0164	 1.6891829238748057
1.03856	 1.6024252075211496
1.06071	 1.507122327330587
1.08287	 1.40206558931345
1.10504	 1.2919807954973652
1.1272	 1.1803886268942598
1.14936	 1.065461756008908
1.17151	 0.954948230863475
1.19367	 0.8484104439659246
1.21584	 0.744999261660289
1.238	 0.6506049701582177
1.26016	 0.561614256986875
1.28231	 0.48210534672817473
1.30447	 0.40958548316144194
1.32664	 0.34445455334580055
1.3488	 0.28896858396510766
1.37096	 0.23892128633065296
1.39312	 0.1955975022715291
1.41528	 0.16151426460477153
1.43744	 0.1299676959165197
1.4596	 0.10449636637928005
1.48176	 0.08416206543128413
1.50392	 0.06691130551507099
};
    \end{axis}
\end{tikzpicture}
  \caption{Estimated density of $\tilde X$, the relative displacement of a cantilever beam.}
\label{fig:canti-density}
\end{figure}
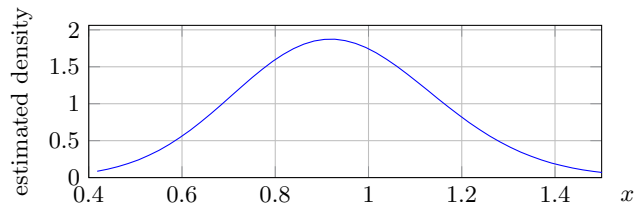

\begin{figure}[!hbtp] %MISE
\centering
 \begin{tikzpicture} \footnotesize  % KDE
   \begin{axis}[ 
		cycle list name=defaultcolorlist,
    legend style={at={(1.3,1.35)}, anchor={north east}},
      xlabel=$\log_2(n)$,
%       ylabel=$\log_2(\MISE)$,
       grid,
       width=6cm,
			 height=4cm,
      ] 
       \addplot table[x=log(n),y=log(MISE)] { 
      log(n)  log(MISE)
      14  -10.90937560535595
      15  -11.67606653490286
      16  -12.44275746444978
      17  -13.20944839399669
      18  -13.97613932354360
      19  -14.74283025309051 
      }; 
      \addlegendentry{MC}
      
      \addplot table[x=log(n),y=log(MISE)] { 
      log(n)  log(MISE) 
      14  -12.96211333219881
      15  -13.86498279702762
      16  -14.76785226185644
      17 -15.67072172668525
      18  -16.57359119151407
      19  -17.47646065634289 
      }; 
      \addlegendentry{Strat}
%       \label{Stratification}
      \addplot table[x=log(n),y=log(MISE)] { 
      log(n)  log(MISE)
      14  -15.69682195516851
      15  -16.67780002050783
      16  -17.65877808584714
      17 -18.63975615118646
      18  -19.62073421652578
      19  -20.60171228186510 
      }; 
      \addlegendentry{Sobol'+LMS}
%       \label{Sobol' + LMS}
      \addplot table[x=log(n),y=log(MISE)] { 
      log(n)  log(MISE) 
      14  -15.712980379532
      15  -16.687189407095
      16  -17.661398434658
      17  -18.635607462221
      18  -19.609816489784
      19  -20.584025517348
      }; 
      \addlegendentry{Sobol'+NUS}
%       \label{Sobol' + NUS}
    \end{axis}
 \end{tikzpicture}
\begin{tikzpicture} \footnotesize  % IV KDE
   \begin{axis}[ 
		cycle list name=defaultcolorlist,
    legend style={at={(1.3,1.35)}, anchor={north east}},
      xlabel=$\log_2(n)$,
%       ylabel=$\log_2(\IV)$,
      grid,
			width=6cm,
			height=4cm,
      ] 
        \addplot table[x=log(n),y=log(IV)] { 
        log(n)  log(IV) 
        14  -10.05334549950600
        15  -11.04389450988958
        16  -12.03444352027316
        17  -13.02499253065674
        18  -14.01554154104032
        19  -15.00609055142390
        }; 
        \addlegendentry{MC}
        \addplot table[x=log(n),y=log(IV)] {
        log(n)  log(IV) 
        14  -12.115131310272240
        15  -13.495009893646634
        16  -14.874888477021027
        17  -16.254767060395421
        18  -17.634645643769814
        19  -19.014524227144208
        }; 
      \addlegendentry{Strat}
%       \label{Stratification}
% 
      \addplot table[x=log(n),y=log(IV)] { 
      log(n)  log(IV) 
      14  -16.14708802034027
      15  -18.08995993207298
      16  -20.03283184380570
      17  -21.97570375553842
      18  -23.91857566727114
      19  -25.86144757900386
      }; 
      \addlegendentry{Sobol'+LMS}
%       \label{Sobol' + LMS}
% 
      \addplot table[x=log(n),y=log(IV)] { 
      log(n)  log(IV) 
      14  -16.17773814021874
      15  -18.10990384092199
      16  -20.04206954162525
      17  -21.97423524232851
      18  -23.90640094303176
      19  -25.83856664373502
      }; 
      \addlegendentry{Sobol'+NUS}
    \end{axis}
\end{tikzpicture}
 \caption{Estimated MISE (left) and IV (right) as a function of $n$ 
   for $h = 2^{-6}$, for the cantilever example. }
 \label{fig:canti-mise}
\end{figure}
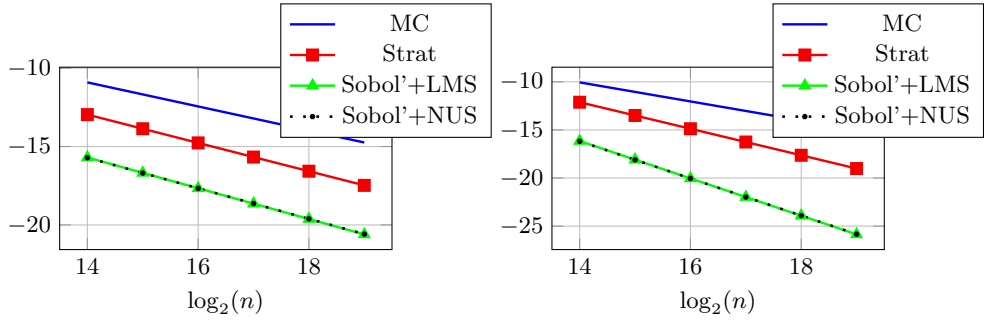

\begin{table}
 \centering
\caption{Experimental results for the KDE, for the displacement of a cantilever beam,
   over the {interval $[0.407,1.515]$.}}
\label{tab:canti-parameters-99} 
\small
\begin{tabular}{l | l l l l}
			& MC	& Strat	& LMS		& NUS\\
\hline
$C$ 			  & 0.109	& 0.022	& 1.8E-4	& 1.5E-4\\
$\beta$			& 0.991	& 1.380	& 1.943		& 1.932\\
$\delta$ 		& 1.168	& 2.113	& 3.922		& 3.933\\
$R^{2}$			& 0.999	& 0.999	& 0.999		& 0.999\\
$B$			    & 107.4	& 107.2	& 107.1		& 107.1\\
\hline 
$\hat{\kappa}_*$	& 0.208	& 0.225	& 0.186		& 0.182\\
$\hat{\gamma}_*$	& 0.192	& 0.226	& 0.245		& 0.244\\
$\ell_*$		      & 5.909	& 6.443	& 7.090		& 7.085\\
$\hat{K}_*$		    & 0.885	& 0.800	& 0.256		& 0.237\\
$\hat{\nu}_*$		  & 0.767	& 0.903	& 0.981		& 0.974\\
\hline
e19			          & 14.74	& 17.48	& 20.60		& 20.58
\end{tabular} 
\end{table}

\Cref{tab:canti-parameters-99} gives the parameter estimates from our experiment.
RQMC increases the rate $\beta$  from 1 to about 2. 
However, $\delta$ increases even more, from 1 to about 4. 
This means that although the variance decreases much faster than for MC as a function of $n$ for fixed $h$,
we cannot afford to decrease $h$ very much to decrease the bias, 
so the MISE reduction is limited.
% The $R^2$ coefficient is very close to 1, showing that the linear model for $\log_2(\IV)$
% fits very well.  \Cref{fig:canti-iv-3d} confirms this. 
% It is reassuring to see that the estimate of $B$ is about the same for all point sets, 
% for both the histogram and the KDE.
RQMC improves both the estimated rate $\hat\nu_*$ and the constant $K$ in the MISE model.

\Cref{fig:canti-mise} shows the estimated MISE as a function of $n$
(with the estimated optimal $h$),
as well as the estimated IV as a function of $n$, all in log scale.
%   for the histogram and the KDE.  
The results for Sobol'+LMS and Sobol'+NUS are
practically indistinguishable in those plots.
We see that although the MISE rate (slope) is not improved much by RQMC,
the MISE is nevertheless reduced by a significant factor. 
With $n=2^{19}$, the MISE is almost $2^{6}=64$ times smaller with Sobol'+LMS than with MC.
\hflorian{In case you wanted to write Sobol'+NUS, it's pretty much the same factor.}
\hflorian{should we stress more that we talk about the estimated MISE and not the 
 true MISE, or do you think that this is clear?}
For fixed $h$, the IV converges at a faster rate with RQMC than with MC.
%  as shown in the lower part of the figure.
\hpierre{It would be nice to see the faster convergence rate of the IV for fixed $h$,
  as we had before. For $h$, you can take perhaps a value close to the optimal $h$ 
	for $n=2^{16}$.  }

Here, $g$ is strictly decreasing in $E$ and strictly increasing in both $X$ and $Y$. 
Therefore, \cref{cor:miseStrat} applies. The asymptotic parameter values are
$\beta=4/3$, $\delta=2$, and $\nu=0.889$, which are very close to what we found 
empirically for stratification (see \cref{tab:canti-parameters-99}).

%%%%%%%%%%%%%%%%%%%%%%%%%%%%%%%%%%%%%%%%%%%%%%%%%%%%%%%%%%%%%%%%%%%%%%%%%%
\subsection{A weighted sum of lognormals}

In this example, we estimate the density of a weighted sum of lognormals: % random variables
$
  X   % = \sum_{j=1}^s w_j S_j 
	  = \sum_{j=1}^s w_j \exp(Y_j)
$
where $\bY = (Y_1,\dots,Y_s)^\tr$ has a multinormal distribution with mean vector $\bmu$
and covariance matrix $\bC$.  Let $\bC = \bA\bA^\tr$ be a decomposition of $\bC$.
To generate $\bY$, we generate $\bZ$ a vector of $s$ independent standard normals by inversion,
then put $\bY = \bmu + \bA\bZ$. 
For MC, the choice of decomposition does not matter, but for RQMC it does,
and here we take the decomposition used in principal component analysis (PCA) \cite{fGLA04a,vLEC09f}.
%  \cite{vACW97a,vLEC09f}.
We also tried sequential sampling (SS) and Brownian bridge sampling (BBS) 
but with them, RQMC did not improve the IV significantly as we will see with PCA.

This model has several applications.  
In one of them, for some positive constants $\rho$ and $s_0$, by taking $w_j = s_0 (s-j+1)/s$, 
$e^{-\rho} \max(X-K,0)$ is the payoff of a financial option based on the average value of a 
stock or commodity price at $s$ observation times, under a geometric Brownian motion process.
Estimating the density of this random payoff in its positive part 
is equivalent to estimating the density of $X$ over the interval $(K,\infty)$
(for simplicity we ignore the scaling factor $e^{-\rho}$).
When we compute the KDE here, the realizations of $X$ smaller than $K$ 
are \emph{not} discarded; they contribute to the KDE slightly above $K$.
Discarding them would introduce a significant bias in the KDE due to a boundary effect at~$K$.

For our numerical experiment, we take this special case with the same parameters as in \cite{vLEC18a}: 
$s=12$, $s_0 = 100$, and $K =101$.
The matrix $\bC$ is defined indirectly as follows.
We have 
$
 Y_j = Y_{j-1} (\mu-\sigma^{2}) j/s + \sigma B(j/s)
$
where $Y_0 = 0$, $\sigma= 0.12136$, $\mu= 0.1$,
and $B(\cdot)$ is a standard Brownian motion.
We estimate the density of $\tilde X = X-K$ over the interval $[a,b] = [0,\, 27.13]$.
Approximately 0.5\% of the density lies on the right of this interval 
and 29.05\% lies on the left (this is when the option brings no payoff).
\Cref{fig:option-density} shows a plot of the estimated density of $\tilde X = X-K$
obtained from a KDE with Sobol'+NUS and $n=2^{19}$ points.

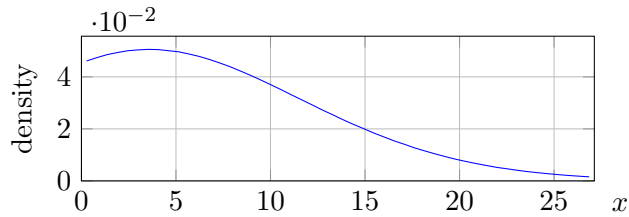
\begin{figure}[!htbp]%density plot cantilever
\centering
 \begin{tikzpicture} 
    \begin{axis}[ 
      x label style={at={(axis description cs:1.05,-0.05)},anchor=north},
      xlabel=$x$,
      ylabel=density,
      xmin=0,
      xmax=27.13,
			ymin=0,
      grid,
      width = 8.4cm,
      height = 3.5cm,
      ] 
      \addplot table[x=x,y=density,color=red,mark= ,style=solid] { 
x 	density
0.2713	 0.04611232690567282
0.8139	 0.04738904939926527
1.3565	 0.04854665641147895
1.89909	 0.04935278005791458
2.4417	 0.05003779663221299
2.9843	 0.05036511319024022
3.5269	 0.05057717669617984
4.06949	 0.05048076499518774
4.6121	 0.05005610817912746
5.1547	 0.04962484608337289
5.69729	 0.04877657435905761
6.2399	 0.047846431481651915
6.7825	 0.04666690897059682
7.3251	 0.045312565641749396
7.86769	 0.043895468554901014
8.4103	 0.0422501939453037
8.9529	 0.04056514795509081
9.4955	 0.038769040020104346
10.0381	 0.036903556535522165
10.5807	 0.035015700886604606
11.12329	 0.03303385710823513
11.66589	 0.031126632588306533
12.2085	 0.029205180049397955
12.7511	 0.027238900533157438
13.2937	 0.02542143420894306
13.8363	 0.023532372619457877
14.3789	 0.021757996814559116
14.9215	 0.020120633022231143
15.46409	 0.01841597424715464
16.00669	 0.016890696049540768
16.54929	 0.015367555093975428
17.0919	 0.014061417020858663
17.6345	 0.012662417554006085
18.1771	 0.011497539468526559
18.7197	 0.01036552130438518
19.2623	 0.00931255531002168
19.8049	 0.008323919530578015
20.3475	 0.007459994626408142
20.8901	 0.0066479772139835846
21.4327	 0.005911735485574319
21.97529	 0.005193871772515958
22.51789	 0.004651090710484343
23.06049	 0.004076965269290676
23.6031	 0.0035709518321854843
24.1457	 0.0031425178338569614
24.6883	 0.002755272862473679
25.2309	 0.002403992975843618
25.7735	 0.002080279965777982
26.3161	 0.00181904699347215
26.8587	 0.001562579253971936
};
    \end{axis}
  \end{tikzpicture}
  \caption{Estimated density of the option payoff $X-K$.}
\label{fig:option-density}
\end{figure}

\begin{table}[hbt]
 \centering
\caption{Experimental results for the density estimation of the option payoff over the interval $[0,\, 27.13]$.}
\label{tab:option-parameters} 
\small
\begin{tabular}{l | l l l }
			& MC		& LMS		& NUS\\
\hline
$C$ 			  & 0.171		& 0.110		& 0.097\\
$\beta$			& 1.005		& 1.671		& 1.663\\
$\delta$ 		& 1.151		& 4.907		& 4.930\\
$R^{2}$			& 0.999		& 0.990		& 0.990\\
$B$			    & 1.1E-6	& 1.1E-6	& 1.1E-6\\
\hline 
$\hat{\kappa}_*$	& 7.953		& 3.717		& 3.657\\
$\hat{\gamma}_*$	& 0.195		& 0.188		& 0.186\\
$\ell_*$		      & 0.715		& 1.670		& 1.668\\
$\hat{K}_*$		    & 0.020		& 3.9E-4	& 3.6E-4\\
$\hat{\nu}_*$		  & 0.780		& 0.750		& 0.745\\
\hline
e19			          & 20.45		& 25.59		& 25.58
\end{tabular}
\end{table}

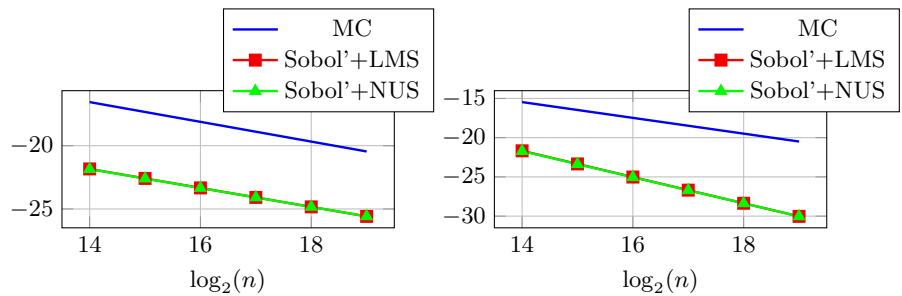
\begin{figure}[!hbtp] \footnotesize  %MISE
\centering
 \begin{tikzpicture} % KDE
   \begin{axis}[ 
		cycle list name=defaultcolorlist,
    legend style={at={(1.12,1.6)}, anchor={north east}},
      xlabel=$\log_2(n)$,
%      ylabel=$\log_2(\MISE)$,
      grid,
			width=6cm,
			height=3.4cm,
      ] 
       \addplot table[x=log(n),y=log(MISE)] { 
      log(n)  log(MISE) 
      14  -16.55333660590604
      15  -17.33362060423514
      16  -18.11390460256424
      17  -18.89418860089335
      18  -19.67447259922245
      19  -20.45475659755155
      }; 
      \addlegendentry{MC}
      
      \addplot table[x=log(n),y=log(MISE)] { 
      log(n)  log(MISE) 
      14  -21.83444970168763
      15  -22.58487944386306
      16  -23.33530918603850
      17  -24.08573892821394
      18  -24.83616867038937
      19  -25.58659841256481
      }; 
      \addlegendentry{Sobol'+LMS}
%       \label{Sobol' + LMS}
      \addplot table[x=log(n),y=log(MISE)] { 
      log(n)  log(MISE) 
      14  -21.85522139071330
      15  -22.60023463862968
      16  -23.34524788654606
      17  -24.09026113446244
      18  -24.83527438237882
      19  -25.58028763029519
      }; 
      \addlegendentry{Sobol'+NUS}
%       \label{Sobol' + NUS}
    \end{axis}
 \end{tikzpicture}
\begin{tikzpicture} % IV KDE
   \begin{axis}[ 
		cycle list name=defaultcolorlist,
    legend style={at={(1.22,1.6)}, anchor={north east}},
      xlabel=$\log_2(n)$,
%       ylabel= $\log_2(\IV)$,
      grid,
			width=6cm,
			height=3.4cm,
      ] 
       \addplot table[x=log(n),y=log(MISE)] { 
      log(n)  log(MISE) 
      14  -15.4676
      15  -16.4724
      16  -17.4772
      17  -18.4819
      18  -19.4867
      19  -20.4915
      }; 
      \addlegendentry{MC}
      
      \addplot table[x=log(n),y=log(MISE)] { 
      log(n)  log(MISE) 
      14  -21.6764
      15  -23.3474
      16  -25.0183
      17  -26.6893
      18  -28.3602
      19  -30.0312
      }; 
      \addlegendentry{Sobol'+LMS}
%       \label{Sobol' + LMS}

      \addplot table[x=log(n),y=log(MISE)] { 
      log(n)  log(MISE) 
      14  -21.7169
      15  -23.3801
      16  -25.0433
      17  -26.7065
      18  -28.3696
      19  -30.0328
      }; 
      \addlegendentry{Sobol'+NUS}
%       \label{Sobol' + NUS}
    \end{axis}
 \end{tikzpicture}
 \caption{Estimated MISE as a function of $n$ (left) and 
    estimated IV as a function of $n$ for $h=1/2$ (right).} % for the option payoff example.}
 \label{fig:option-mise}
\end{figure}

\Cref{tab:option-parameters} summarizes the results of our experiments.
Again, the linear model for the IV fits extremely well in the selected area.
RQMC improves $\beta$ from 1 to about $5/3$, which is significant,
but at the same time $\delta$ increases from about 1.1 to nearly 5.
This means we are very limited in how much we can decrease $h$ to reduce the bias.
On the other hand, this empirical $\delta$ is not as bad as the one in the AIV bound 
of \cref{cor:miseHK}, which gives $\delta = 2s = 24$. 
The estimate of $B$ is again about the same for all point sets.
% , which is reassuring.
Somewhat surprisingly, in the region considered, 
the estimated MISE rate $\hat\nu_*$ is not better for
RQMC than for MC, due to the large $\delta$,
but the MISE is nevertheless about 32 times smaller for RQMC than for MC in the 
range of interest,
as shown in \Cref{fig:option-mise}, for which $h$ was taken 
as the estimated optimal $h$ from our model, as a function of $n$.
That is, RQMC is truly beneficial for estimating the payoff density in this 12-dimensional example.
In the lower panel, we see that the estimated IV for fixed $h$
converges faster with RQMC than with MC.
% For $h = 1/2$ and $n = 2^{20}$, for example, we have $\log_2(\IV) \approx -22.4$ for MC
% and $\log_2(\IV)\approx -33.5$ for Sobol'+NUS, so the IV is reduced by a factor of about 2000.

\hpierre{When estimating the mean $\EE[X]$ instead of the density, 
the RQMC improvement is more spectacular. 
For comparison, with Sobol'+LMS, the variance converges approximately as $\cO(n^{-1.9})$ 
compared with $\cO(n^{-1})$ for MC,
the variance is divided by a factor of about two millions compared with MC for $n = 2^{20}$,
and there is no bias.  See \cite{vLEC18a}, Table 3.}

%%%%%%%%%%%%%%%%%%%%%%%%%%%%%%%%%%%%%%%%%%%%%%%%%%%%%%%%%%%%%%%
\section{Conclusion}
\label{sec:conclusion}

We explored RQMC combined with KDEs to estimate a density by simulation.  
RQMC can improve the IV and the MISE, sometimes by large factors, in situations
in which the (effective) dimension is small.
The improvement is more limited when the effective dimension is large.
We also found that the IV improvement degrades quickly as a function of $h$ when $h\to 0$.
In our empirical experiments, the IV was never larger with KDE+RQMC 
than with KDE+MC, and was often much smaller.
\hflorian{I have checked the data of the first examples and I have not found one single instance, where the IV with RQMC is larger than with MC. The other examples are in the paper and I haven't found anything there either.}

\hpierre{In the online supplement, we report a similar analysis for histograms instead of KDEs,
and find that RQMC also brings some improvement, but more limited than with the KDE.
We also provide additional experimental results and details.}

% Sometimes, the IV decreases at a faster rate with RQMC than MC as a function of $n$
% for fixed $h$, but this rate degrades when we decrease $h$ to reduce the ISB,
% and as a result the MISE often decreases not much faster with RQMC than with MC.
\hpierre{-- Must be completed with more positive statements....   Summarize our bounds.
  Summarize empirical results.  Suggestions for further work?  }

%
\iffalse  %%%%%%%
An alternative class of approaches that we plan to study next is 
via gradient estimators  \cite{vASM18a,sASM07a,oLEC90a,oLEC91a,vLEC94b}.
The idea is to obtain a smooth estimator of the cdf of $X$ and take the sample
derivative of that estimator to estimate the density.
% Stay tuned.
\fi  %%%%%%%%

%%%%%%%%%%%%%%%%%%%%%%%%%%%%%%%%%%%%%%%%%%%%%%%%%
\section*{Acknowledgments}

%  THE GRANTS ARE IN FOOTNOTE AT THE BEGINNING.

The idea of this work started during a workshop at the Banff International Research 
Station (BIRS) in October 2015.  Part of the research was accomplished within a research
program on quasi-Monte Carlo sampling methods at the 
Statistical and Applied Mathematical Sciences Institute (SAMSI), in North Carolina, 
in 2017--2018.  We thank Ilse Ipsen for her support in organizing this program.

% \bibliographystyle{siamplain}

% \bibliography{ift,math,optim,random,simul,stat,vrt,fin}
%  \bibliography{optim,random,simul,stat,vrt}
\end{document}